\documentclass[aop,preprint]{imsart}

\RequirePackage[OT1]{fontenc}
\RequirePackage{amsthm,amsmath,amsfonts}
\RequirePackage[numbers]{natbib}
\usepackage{graphicx}
\usepackage{amssymb}
\usepackage{multirow,array}
\usepackage{comment}
\usepackage{enumitem}
\usepackage{emptypage}
\usepackage{color}
\usepackage{mathrsfs}
\usepackage{url}

\usepackage[font={small}]{caption}
\usepackage{float}

\usepackage{pgfplots}
\pgfplotsset{compat=1.6}

\pgfplotsset{soldot/.style={color=blue,only marks,mark=*}} \pgfplotsset{holdot/.style={color=blue,fill=white,only marks,mark=*}}
\pgfplotsset{soldotred/.style={color=red,only marks,mark=*}} \pgfplotsset{holdotred/.style={color=red,fill=white,only marks,mark=*}}

\arxiv{arXiv:1704.08053}

\startlocaldefs
\newcounter{dummy}
\numberwithin{dummy}{section}

\newtheorem{theorem}[dummy]{Theorem}

\newtheorem{assumption}[dummy]{Assumption}

\newtheorem{corollary}[dummy]{Corollary}

\newtheorem{definition}[dummy]{Definition}

\newtheorem{lemma}[dummy]{Lemma}

\newtheorem{proposition}[dummy]{Proposition}

\theoremstyle{remark}
\newtheorem{remark}[dummy]{Remark}

\theoremstyle{definition}
\newtheorem{example}[dummy]{Example}

\newcommand{\Lip}{\textnormal{Lip}}

\newcommand{\calP}{\mathcal{P}}
\newcommand{\DD}{\mathscr{D}}
\newcommand{\D}{\mathcal{D}}

\newcommand{\FF}{\mathcal{F}}

\newcommand{\R}{\mathbb R}

\newcommand{\Z}{\mathbb Z}

\newcommand{\X}{\mathbf X}
\newcommand{\Y}{\mathbf Y}
\newcommand{\x}{\mathbf x}
\newcommand{\y}{\mathbf y}

\newcommand{\ZZ}{\mathbf Z}

\newcommand{\Nbf}{\mathbf N}
\newcommand{\M}{\mathbf M}

\newcommand{\var}[1]{#1\textnormal{-var}}

\newcommand{\pvar}{{p\textnormal{-var}}}
\newcommand{\pprimevar}{p'\textnormal{-var}}
\newcommand{\qvar}{q\textnormal{-var}}
\newcommand{\onevar}{1\textnormal{-var}}

\newcommand{\oneHol}{1\textnormal{-H{\"o}l}}

\newcommand{\EEE}[1]{\mathbb E \left[ #1 \right]}

\newcommand{\PP}{\mathbb{P}}

\newcommand{\E}{\mathbb{E}}

\newcommand{\Diff}{\textnormal{Diff}}

\newcommand{\PPP}[1]{\mathbb{P} \left[ #1 \right]}

\newcommand{\id}{\textnormal{id}}
\newcommand{\norm}[1]{\left\|#1 \right\|}

\newcommand{\norms}[1]{\left|#1 \right|}
\newcommand{\normc}{\norm{\cdot}}

\newcommand{\1}[1]{\mathbf 1_{\{#1\}}}

\newcommand{\LLL}{\mathbf L}

\newcommand{\plus}{\mathbf{plus}}

\newcommand{\g}{\mathfrak g}

\newcommand{\floor}[1]{\lfloor #1 \rfloor}

\newcommand{\midset}{\,\middle|\,}

\begin{document}

\begin{frontmatter}
\title{Canonical RDEs and general semimartingales as rough paths} 
\runtitle{Canonical RDEs and semimartingales}

\begin{aug}
\author{\fnms{Ilya} \snm{Chevyrev}\thanksref{t1,m1}
\ead[label=e1]{chevyrev@maths.ox.ac.uk}}
\and
\author{\fnms{Peter K.} \snm{Friz}\thanksref{t2,m2,m3}
\ead[label=e2]{friz@math.tu-berlin.de}}

\thankstext{t1}{Affiliated to TU Berlin  and supported by DFG research unit FOR2402 when this project was commenced. Currently supported by a Junior Research Fellowship of St John's College, Oxford.}
\thankstext{t2}{Partially supported by the European Research Council through CoG-683166 and DFG research unit FOR2402}
\runauthor{I. Chevyrev and P. K. Friz}

\affiliation{University of Oxford\thanksmark{m1}, TU Berlin\thanksmark{m2} and WIAS\thanksmark{m3}}

\address{I. Chevyrev\\
Mathematical Institute\\
University of Oxford\\
Andrew Wiles Building\\
Radcliffe Observatory Quarter\\
Woodstock Road\\
Oxford OX2 6GG\\
United Kingdom\\
\printead{e1}}

\address{P. K. Friz\\
Institut f\"ur Mathematik\\
Technische Universit\"at Berlin\\
Strasse des 17. Juni 136\\
10623 Berlin\\
Germany\\
and\\
Weierstra\ss --Institut f\"ur Angewandte\\
\quad Analysis und Stochastik\\
Mohrenstrasse 39\\
10117 Berlin\\
Germany\\
\printead{e2}}


\end{aug}

\begin{abstract}
In the spirit of Marcus canonical stochastic differential equations, we study a similar notion of rough differential equations (RDEs), notably dropping the assumption of continuity prevalent in the rough path literature. A new metric is exhibited in which the solution map is a continuous function of the driving rough path {\it and} a so-called path function, which directly models the effect of the jump on the system.  In a second part, we show that general multidimensional semimartingales admit canonically defined rough path lifts. An extension of L{\'e}pingle's BDG inequality to this setting is given, and in turn leads to a number of novel limit theorems for semimartingale driven differential equations, both in law and in probability, conveniently phrased via Kurtz--Protter's uniformly-controlled-variations (UCV) condition. A number of examples illustrate the scope of our results.
\end{abstract}



\end{frontmatter}

\section{Introduction}

It{\^o} stochastic integrals are well-known to violate a first order chain rule of Newton--Leibniz type, as is manifest from It{\^o}'s formula. In a number of applications, is is important
to have a chain rule which,
in the context of {\it continuous} semimartingales, was achieved in a satisfactory way by Stratonovich stochastic integration, which - loosely speaking - replaces left-point evaluation (in It{\^o}--Riemann sums) by a symmetric mid-point evaluation. In the case of stochastic integration against {\it general} semimartingales (L{\'e}vy processes as an important special case), one can check that the Stratonovich integral no longer gives a chain rule -- a more sophisticated approach is necessary to take care of jumps and the mechanism for doing this was developed by Marcus~\cite{Mar78, Mar81}. The resulting ``Marcus canonical integration'' and ``Marcus canonical (stochastic differential) equations'' (in the terminology of~\cite{Applebaum09}) was then investigated in a number of works, including~\cite{FK85, K90, F91, AT01, AK93, KPP95, K96, FK99}, see also~\cite{Applebaum09,CP14} and the references therein.

\smallskip

On the other hand, continuous  stochastic integration has been understood for some time in the context of rough path theory, see~\cite{Lyons98,CL05,FrizVictoir08},~\cite[Ch.~14]{FrizVictoir10}. Loosely speaking, given a multidimensional continuous semimartingale $X$, the Stratonovich integral $\int f(X) \circ dX$ can be given a robust (pathwise) meaning in terms of $\X = (X, \int X \otimes \circ dX)$, a.e. realization of which constitutes a geometric rough path of finite $p$-variation for any $p>2$. In contrast to the popular class of H\"older rough paths (usually sufficient to deal with Brownian motion, see e.g.~\cite{FrizHairer14}), $p$-variation has the advantage that it immediately allows for jumps. This also prompts the remark that Young theory, somewhat the origin of Lyons' rough paths, by no means requires continuity. Extensions of rough path theory to a general $p$-variation setting (for possibly discontinuous paths) were then
explored in~\cite{Williams01},~\cite{FrizShekhar17} and~\cite{Chevyrev15}. However, none of these works provided a proper extension of Lyons' main result in rough path analysis: continuity of the solution map as a function of the driving rough path.

\smallskip

The first contribution of this paper is exactly that. We introduce a new metric on the space of c{\`a}dl{\`a}g rough paths, and a type of (Marcus) canonical rough differential equation, for which one has the desired stability result. (Experts in {\it la th{\'e}orie g{\'e}n{\'e}rale des processus} will recognize our topology as a $p$-variation rough paths variant of Skorokhod's strong $M_1$ topology.) In fact, we reserve the prefix ``Marcus'' to situations in which jumps only arise in the $d$-dimensional driving signal, and are handled (in the spirit of Marcus) by connecting $X_{t-}$ and $X_t$ by a straight line. (As a straight line has no area, this creates no jump in the area.) A ``general'' rough path (level $N$, over $\R^d$), however, can have jumps of arbitrary value $\X_{t-}^{-1} \otimes \X_t \in G^N (\R^d)$, and there are many (different) ways to implement Marcus's idea of continuously (parametrized over a fictitious time interval) connecting $\X_{t-}$ and $\X_t$. This is really a modelling choice, no different than choosing the driving signal and/or the driving vector fields. The notion of a ``path function'' $\phi$ helps us to formalize this, and indeed one may view $(\X,\phi)$ as the correct/extended rough driver.

\smallskip

In the second part of the paper, we show how general (c{\`a}dl{\`a}g) semimartingales fit into the theory. In particular, we show that the canoncial lift of a semimartingale indeed is a.s. a (geometric) rough path of finite $p$-variation for any $p>2$ (several special cases, including L{\'e}vy processes, were discussed in \cite{Williams01, FrizShekhar17} but the general case remained open). Our result is further made quantitative by establishing a BDG inequality for general local martingale rough paths. (We thus extend simultaneously the classical $p$-variation BDG inequality~\cite{Lepingle76}, and its version for continuous local martingale rough paths~\cite{FrizVictoir08}). This BDG inequality turns out to be a powerful tool, especially in conjunction with {\it uniform tightness} (UT) and {\it uniformly controlled variation} (UCV) of semimartingale sequences. (Introduced by~\cite{JMP89} and~\cite{KP91} respectively, these conditions are at the heart of basic convergence theorems for stochastic integrals in Skorokhod topology; we work only with UCV in this article, but note that UCV and UT are equivalent under extra assumptions, e.g., convergence in law~\cite[Thm~7.6]{KP96}. See Section~\ref{subsec:convLoc} for the definition of UCV, and also~\cite{CL05} for some links to continuous semimartingale rough paths.)

\smallskip

As an example of an application to general semimartingale theory, we are able to state a criterion for convergence in law (resp. in probability) of Marcus SDEs, which is an analogue of the celebrated criterion for It{\^o} SDEs due to Kurtz--Protter~\cite[Thm.~5.4]{KP91} (we emphasize however that neither criterion is a simple consequence of the other). Loosely speaking, the result asserts that if $X, (X^n)_{n \geq 1}$ are $\R^d$-valued semimartingales such that $(X^n)_{n \geq 1}$ satisfies UCV and $X^n \rightarrow X$ in law (resp. in probability) for the Skorokhod topology, then the solutions to Marcus SDEs driven by $X^n$ (along fixed vector fields) converge in law (resp. in probability) to the Marcus SDE driven by $X$ (see Theorem~\ref{thm:MarcusSDEs} for a precise formulation).
Our theorem (which crucially involves rough paths in the proof, but not in the statement) entails a pleasantly elegant approach to the Wong--Zakai theorem for SDEs with jumps (Kurtz--Protter--Pardoux~\cite{KPP95}, with novel interest from physics~\cite{CP14}) and in fact gives a number of novel limit theorems for Marcus canonical SDEs (see Theorem~\ref{thm:WongZakai}). We remark further that homeomorphism and diffeomorphism properties of solution flows are straightforward, in contrast to rather lengthy and technical considerations required in a classical setting (see, e.g.,~\cite[p.~423]{Applebaum09} and the references therein). At last, we discuss the impact of more general path functions, noting that the ``Marcus choice'' really corresponds to the special case of the linear path function.

\smallskip

The paper is organized as follows. In Section~\ref{sec:RPwJ} we collect some necessary preparatory material, including basic properties of path functions. In Section~\ref{sec:canonicalRDEs} we give meaning to canonical RDEs, for which drivers are (rough path, path function) pairs $(\X,\phi)$, and introduce the metric $\alpha_{\pvar}$ for which the direct analogue of Lyons' universal limit theorem holds.
Section~\ref{sec:cadlagSemimart} is then devoted to applications to c{\`a}dl{\`a}g semimartingale theory, particularly in connection with the UCV condition and Wong--Zakai type approximations. We briefly comment in Section~\ref{sec:Beyond} on the further scope of the theory.

\smallskip

\section{Preparatory material}\label{sec:RPwJ}
\subsection{Wiener and Skorokhod space}
Throughout the paper, we denote by $C([s,t],E)$ and $D([s,t],E)$ the space of continuous and c{\`a}dl{\`a}g functions (paths) respectively from an interval $[s,t]$ into a metric space $(E,d)$. 

Unless otherwise stated, we equip $C([s,t],E)$ and $D([s,t],E)$ respectively with the uniform metric
\[
d_{\infty;[s,t]}(\x,\bar\x) = \sup_{u\in [s,t]}d(\x_u,\bar\x_u)
\]
and the Skorokhod metric
\[
\sigma_{\infty;[s,t]} (\x,\bar\x) = \inf_{\lambda \in \Lambda_{[s,t]}} |\lambda| \vee d_{\infty;[s,t]}(\x\circ \lambda,\bar\x),
\]
where $\Lambda_{[s,t]}$ denotes the set of all strictly increasing bijections of $[s,t]$ to itself, and $|\lambda| := \sup_{u \in [s,t]}|\lambda(u) - u|$.   
When we omit the interval $[s,t]$ from our notation, we will always assume it is $[0,T]$.  We recall that if $E$ is Polish, then so is the Skorokhod space $D([0,T],E)$ with topology induced by $\sigma_\infty$, also known as the $J_1$-topology.

We always let $\D = (t_0 = s < t_1 < \ldots < t_{k-1} < t_k = t)$ denote a partition of $[s,t]$, and notation such as $\sum_{t_i \in \D}$ denotes summation over all points in $\D$ (possibly without the initial/final point depending on the indexing). We let $|\D| = \max_{t_i \in \D} |t_{i+1}-t_i|$ denote the mesh-size of a partition.

For $p > 0$, we define the $p$-variation of a path $\x \in D([s,t],E)$ by
\[
\norm{\x}_{\pvar;[s,t]} := \sup_{\D \subset [s,t]} \Big( \sum_{t_i \in \D} d(\x_{t_i},\x_{t_{i+1}})^p \Big)^{1/p}.
\]
We use superscript notation such as $D^{\pvar}([s,t],E)$ to denote subspaces of paths of finite $p$-variation. For continuous $\x$ only $p \ge 1$ is interesting, for otherwise $\x$ is constant.

\begin{remark}
At least when $E = \R^d$ (or $G^N(\R^d)$, see below) there is an immediate $p$-variation metric and topology. Due to the fact that convergence in $J_1$ topology to a continuous limit is equivalent to uniform convergence, a discontinuous path cannot be approximated by a sequence of continuous paths in the metric $\sigma_\infty$. The same will be true for a $J_1$/$p$-variation (rough path) metric $\sigma_{\pvar}$ below. That said, we will propose below a useful $SM_1$/$p$-variation (rough path) metric $\alpha_{\pvar}$ under which the space of continuous rough paths is not closed.
\end{remark}

\subsection{Rough paths}\label{subsec:RPs}

For $N \geq 1$, we let $G^N(\R^d) \subset T^N(\R^d) \equiv \sum_{k=0}^N (\R^d)^{\otimes k}$ denote the step-$N$ free nilpotent Lie group over $\R^d$, embedded into the truncated tensor algebra ($T^N(\R^d),+,\otimes)$, which we equip with the Carnot-Carath{\'e}odory norm $\normc$ and the induced (left-invariant) metric $d$. Recall that the step-$N$ free nilpotent Lie algebra $\mathfrak{g}^{N}(\R^d) = \log G^N (\R^d)  \subset T^N(\R^d)$ is the Lie algebra of $G^N(\R^d)$. The space $C^{\pvar}([0,T],G^N(\R^d))$, with $N=\floor p$, is the classical space of (continuous, weakly) geometric $p$-rough paths as introduced by Lyons.

Unless otherwise stated, we always suppose a path $\x : [s,t] \rightarrow G^N(\R^d)$ starts from the identity $\x_s = 1_{G^N(\R^d)}$. We denote the increments of a path by $\x_{s,t} = \x^{-1}_s\x_t$. We consider on $D^{\pvar}([s,t], G^N(\R^d))$ the inhomogeneous $p$-variation metric
\begin{equation}\label{eq:rhoDef}
\rho_{\pvar;[s,t]}(\x,\bar\x) = \max_{1 \leq k \leq N} \sup_{\D \subset [s,t]} \Big(\sum_{t_i \in \D} \norms{\x_{t_i,t_{i+1}}^k - \bar\x_{t_i,t_{i+1}}^{k}}^{p/k}\Big)^{k/p}.
\end{equation}
Unless otherwise stated, we shall always assume that $p$ and $N$ satisfy $\floor p \leq N$.

We let $C^{0,\pvar}([s,t], G^N(\R^d))$ denote the closure in $C^{\pvar}([s,t], G^N(\R^d))$ under the metric $\rho_{\pvar}$ of the lifts of smooth paths $C^\infty([s,t],\R^d)$. Recall in particular that $C^{0,\onevar}([s,t], G^N(\R^d))$ is precisely the space of absolutely continuous paths.

We let $V = (V_1,\ldots, V_d)$ denote a collection of vector fields in  $\Lip^{\gamma+m-1}(\R^e)$ with $\gamma > p$ and $m \geq 1$. For a geometric $p$-rough path $\x \in C^{\pvar}([s,t], G^N(\R^d))$, we let $\pi_{(V)}(s,y_s;\x) \in C^{\pvar}([s,t],\R^e)$ denote the solution to the RDE
\[
dy_t = V(y_t) d\x_t, \; \; y_s \in \R^e.
\]
We let $U^\x_{t\leftarrow s} : \R^e \to \R^e$ denote the associated flow map $y \mapsto \pi_{(V)}(s,y;\x)_t$, which we recall is an element of $\Diff^{m}(\R^e)$.
For further details on the theory of (continuous) rough paths theory, we refer to~\cite{FrizVictoir10}.

For the purpose of his paper we have (cf.~\cite{FrizShekhar17})

\begin{definition}\label{def:RPs}
Let $1 \leq p < N+1$. Any $\x \in D^{\pvar}([0,T],G^{N}(\R^d))$ is called a \textit{general (c\`adl\`ag, weakly) geometric $p$-rough path over } $\mathbb{R}^d$. Define $\Delta \x_t := \x_{t-}^{-1} \otimes \x_t$ and say $\x$ is Marcus-like if 
for all $t \in [0,T]$
\[
\log \Delta \x_t \in \mathbb{R}^{d}\oplus \left\{ 0\right\} \oplus ... \oplus \left\{ 0\right\}  \subset \mathfrak{g}^{N}(\mathbb{R}^{d}),
\]
where $\log$ is taken in $T^N(\R^d)$.
\end{definition}
As we will see later, any canonical lift of a general $d$-dimensional semimartingale $X$ (with area given by $\tfrac{1}{2} \int [ X^-, dX]$) gives rise to a Marcus-like general geometric $p$-rough path for $p>2$. The model case
of L{\'e}vy processes was studied in~\cite{Williams98, FrizShekhar17}.

\subsection{Path functions}
\label{subsec:pathFuncs}

We briefly review and elaborate on the concept of a path function introduced in~\cite{Chevyrev15}. Let $(E,d)$ be a metric space.

\begin{definition}
A path function on $E$ is a map $\phi : J \to C([0,1],E)$ defined on a subset $J \subseteq E\times E$ for which $\phi(x,y)_0 = x$ and $\phi(x,y)_1 = y$ for all $(x,y) \in J$.

For a path $\x \in D([0,T], E)$, we say that $t \in [0,T]$ is a jump time of $\x$ if $\x_{t-} \neq \x_t$. We call the pair $(\x,\phi)$ admissible if $(\x_{t-},\x_t) \in J$ for all jumps times $t$ of $\x$.
We say that two admissible pairs $(\x,\phi)$ and $(\bar\x,\bar\phi)$ are equivalent, and write $(\x,\phi) \sim (\bar \x,\bar \phi)$, if $\x = \bar \x$ and $\phi(\x_{t-},\x_t)$ is a reparametrization of $\bar\phi(\x_{t-},\x_t)$.

We denote by $\bar\DD([0,T],E)$ the set of all admissible pairs $(\x,\phi)$, and by $\DD([0,T],E) = \bar\DD([0,T],E)/\sim$ the set of all equivalence classes of admissible pairs.
For a fixed path function $\phi$, let $\DD_\phi([0,T],E)$ denote the set of all $\x \in D([0,T], E)$ such that $(\x,\phi)$ is admissible.
\end{definition}

We will often simply say that $\phi$ is a path function on $E$ and keep implicit the fact there is an underlying domain of definition $J$.
We point out that situations where $J \neq E\times E$ arise naturally when studying solution maps of canonical c{\`a}dl{\`a}g RDEs, see Theorem~\ref{thm:alphaMetric} and the discussion before it.

In the case that $E$ is a Lie group with identity element $1_E$ (taken in this article to always be $G^N(\R^d)$), we shall often assume that $\phi$ is \emph{left-invariant}, which is to say that there exists a subset $B \subseteq E$ such that $J = \left\{(x,y) \in E\times E \midset x^{-1}y \in B\right\}$ and
\[
\phi(x,y)_t = x\phi(1_E,x^{-1}y)_t, \; \forall (x,y) \in J, \; \forall t \in [0,1].
\]
In this case, it is equivalent to view $\phi$ as a map $\phi : B \to C([0,1],E)$ such that $\phi(x)_0 = 1_E$ and $\phi(x)_1 = x$ for all $x \in B$, for which $\phi(x,y)_t = x\phi(x^{-1}y)_t$. Whenever we write $\phi$ with only one argument as $\phi(x)$, we shall always mean that it is left-invariant.

\begin{example}[log-linear and Marcus path function] \label{ex:logLinear}
The prototypical example of a (left-invariant) path function $\phi$ on $G^N(\R^d)$, which we shall often refer to in the paper, is the \emph{log-linear} path function
\[
\phi(x)_t = e^{t\log x}, \; \; \forall x \in G^N(\R^d), \; \forall t \in [0,1],
\]
where $\log$ is taken in $T^N(\R^d)$.
Since $J = G^N(\R^d) \times G^N(\R^d)$, we have $(\x,\phi)$ is admissible for all $\x \in D([0,T],G^N(\R^d))$. When $N=1$ we see a familiar special case: since $G^1(\R^d) \cong \R^d$, one has $\phi(x)_t = t x$, and then $\phi (x,y) = x + t(y - x).$ This is precisely the ``Marcus interpolation'' of a c{\`a}dl{\`a}g path before and after its jump. See parametric plots in Figures~\ref{fig:cadlag} and~\ref{fig:Marcus}.  
\end{example}

\begin{figure}
\centering
\begin{minipage}{0.45\textwidth}
\centering
\caption{$2$-dimensional c\`adl\`ag path} \label{fig:cadlag}
\smallskip
\begin{tikzpicture}[scale=0.75]
\begin{axis}[
  xtick={0,4,8}, ytick={0,4,8},  xticklabels={,,},  yticklabels={,,},  xlabel={$x$}, ylabel={$y$}, xlabel style={below right}, ylabel style={above left},
  xmin= -3, xmax=10,  ymin=-3,  ymax=10]]
  
\addplot[domain=0:3,blue, thick] ({-1+x*3},{7-(x-3)^2});   
\addplot[domain=3:6,blue, thick, samples = 200] ({ 7- 2*(x-3)},{5 + sin(  (x-3)/(6-3) * 3 *360 ) }); 
\addplot[domain=6:10,blue, thick, samples= 200] ({0 - (x-6)/2},{8-(x-6)/2 - sin(  (x-6)/(10-6) * 1 *360 ) / 3 });
\draw[dotted, blue,->] (axis cs:8,7) -- (axis cs:7,5);
\draw[dotted, blue] (axis cs:1,5) -- (axis cs:0,8);

\addplot[soldot] coordinates{(7,5)(0,8)};   
\addplot[holdot] coordinates{(8,7)(1,5)}; 

\end{axis}
\end{tikzpicture}
\end{minipage}
\hfill
\begin{minipage}{0.45\textwidth}
\centering
\caption{Marcus interpolation} \label{fig:Marcus}
\smallskip
\begin{tikzpicture}[scale=0.75]
\begin{axis}[
  xtick={0,4,8}, ytick={0,4,8},  xticklabels={,,},  yticklabels={,,},  xlabel={$x$}, ylabel={$y$}, xlabel style={below right}, ylabel style={above left},
  xmin= -3, xmax=10,  ymin=-3,  ymax=10]]

\addplot[domain=0:3,blue, thick] ({-1+x*3},{7-(x-3)^2});   
\addplot[domain=3:6,blue, thick, samples = 200] ({ 7- 2*(x-3)},{5 + sin(  (x-3)/(6-3) * 3 *360 ) }); 
\addplot[domain=6:10,blue, thick, samples= 200] ({0 - (x-6)/2},{8-(x-6)/2 - sin(  (x-6)/(10-6) * 1 *360 ) / 3 });

\draw[ultra thick, red] (axis cs:8,7) -- (axis cs:7,5);
\draw[ultra thick, red] (axis cs:1,5) -- (axis cs:0,8);

\addplot[soldotred] coordinates{(7,5)(0,8)(8,7)(1,5)};   
\end{axis}
\end{tikzpicture}
\end{minipage}
\end{figure}

For $(\x,\phi) \in \bar\DD([0,T],E)$ we now construct a continuous path $\x^{\phi} \in C([0,T], E)$ as follows. Fix a convergent series of strictly positive numbers $\sum_{k=1}^\infty r_k$. Let $t_1, t_2, \ldots$ be the jump times of $\x$ ordered so that $d(\x_{t_1-},\x_{t_1}) \geq d(\x_{t_2-},\x_{t_2}) \geq \ldots$, and $t_j < t_{j+1}$ if $d(\x_{t_j-},\x_{t_j}) = d(\x_{t_{j+1}-},\x_{t_{j+1}})$. Let $0 \leq m \leq \infty$ be the number of jumps of $\x$.

Let $r = \sum_{k=1}^m r_k$ and define the strictly increasing (c{\`a}dl{\`a}g) function
\[
\tau : [0,T] \to [0, T+r], \; \; \tau(t) = t + \sum_{k=1}^m r_{k} \1{t_k \leq t}.
\]
Note that $\tau(t-) < \tau(t)$ if and only if $t = t_k$ for some $1 \leq k < m+1$. Moreover, note that the interval $[\tau(t_k-), \tau(t_k))$ is of length $r_{k}$.

Define $\widehat \x \in C([0,T+r],E)$ by
\[
\widehat \x_s =
\begin{cases} \x_t &\mbox{if $s = \tau(t)$ for some $t \in [0,T]$} \\ 
\phi(\x_{t_k-}, \x_{t_k})_{(s-\tau(t_k-))/r_{k}} &\mbox{if $s \in [\tau(t_k-), \tau(t_k))$ for some $1 \leq k < m+1$}.
\end{cases}
\]
Denote by $\tau_r(t) = t(T+r)/T$ the increasing linear bijection from $[0,T]$ to $[0,T+r]$. We finally define
\begin{equation*}\label{def:xphi}
\x^\phi = \widehat \x \circ \tau_r \in C([0,T], E).
\end{equation*}
We note that one can recover $\x = \x^\phi \circ \tau_\x$ via the time change
\begin{equation}\label{eq:timeChange}
\tau_\x := \tau_r^{-1} \circ \tau,
\end{equation}
for which it holds that
\begin{equation*} \label{littleest}
\sup_{t \in [0,T]}|\tau_\x(t) - t| \leq \sum_{k=1}^\infty r_k.
\end{equation*}

\begin{figure}
\centering
\caption{Hoff interpolation $\x^\phi$, with $\phi$ given as in Example~\ref{ex:Hoff}} \label{fig:Hoff}
\smallskip
\begin{tikzpicture}[scale=0.75]
\begin{axis}[
xtick={0,4,8}, ytick={0,4,8},  xticklabels={,,},  yticklabels={,,},  xlabel={$x$}, ylabel={$y$}, xlabel style={below right}, ylabel style={above left},
  xmin= -3, xmax=10,  ymin=-3,  ymax=10]]

\addplot[domain=0:3,blue, thick] ({-1+x*3},{7-(x-3)^2});   
\addplot[domain=3:6,blue, thick, samples = 200] ({ 7- 2*(x-3)},{5 + sin(  (x-3)/(6-3) * 3 *360 ) }); 
\addplot[domain=6:10,blue, thick, samples= 200] ({0 - (x-6)/2},{8-(x-6)/2 - sin(  (x-6)/(10-6) * 1 *360 ) / 3 });

\draw[ultra thick, red] (axis cs:8,7) -- (axis cs:7,7);
\draw[ultra thick, red] (axis cs:7,7) -- (axis cs:7,5);
\draw[ultra thick, red] (axis cs:1,5) -- (axis cs:0,5);
\draw[ultra thick, red] (axis cs:0,5) -- (axis cs:0,8);

\addplot[soldotred] coordinates{(7,5)(0,8)(8,7)(1,5)}; 
\end{axis}
\end{tikzpicture}
\end{figure}
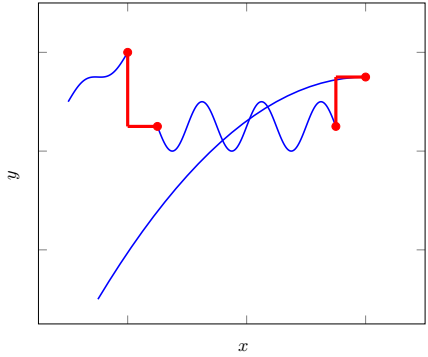

\begin{remark}[Intrinsic definition of $\x^\phi$] \label{rmk:xPhiwelldefined} The construction of $\x^\phi$ involves an ad-hoc choice, namely the sequence $(r_n)$ and the increasing bijection $\tau_r$. If $\bar \x^\phi$ is constructed similarly, but via a sequence $(\bar r_n)$, followed by another reparametrisation given by $\bar \tau_{\bar r}$, then $\x^\phi$ and $\bar \x^\phi$ are reparametrizations of one another.
\end{remark}

\begin{remark}
The construction above is similar to ones appearing in~\cite{FrizShekhar17, Williams01}, and is a simplification of the construction in~\cite{Chevyrev15}. The primary difference is that in~\cite{Chevyrev15} the added fictitious time $r_k$ for the jump $t_k$ depended further on the size of the jump $d(\x_{t_k-},\x_{t_k})$. This extra dependence was used to show continuity of the map $\x \mapsto \x^\phi$ from $D([0,T],E) \to C([0,T],E)$, which we will not require here.
\end{remark}

\subsection{A generalisation of Skorokhod's $SM_1$ topology}
\label{subsec:SkorSM1}

For $(\x,\phi) \in \bar\DD([0,T],E)$ and $\delta > 0$, let $\x^{\phi,\delta} \in C([0,T],E)$ be constructed in the same procedure as $\x^\phi$ but using the series $\sum_{k=1}^\infty \delta r_k$ instead of $\sum_{k=1}^\infty r_k$.

\begin{lemma}\label{lem:limExists}
For all $(\x,\phi), (\bar \x, \bar \phi) \in \bar\DD([0,T],E)$, the limit 
\begin{equation}\label{eq:limSigma}
\lim_{\delta \rightarrow 0} \sigma_{\infty;[0,T]} (\x^{\phi,\delta},\bar \x^{\bar\phi,\delta})
\end{equation}
exists, is independent of the choice of series $\sum_{k=1}^\infty r_k$, and induces a pseudometric on the set of equivalence classes $\DD([0,T],E)$.
\end{lemma}

\begin{proof}
To show that the limit exists, note that for every $\delta, \bar\delta > 0$, there exists $\lambda \in \Lambda$ such that $|\lambda| < 2(\delta + \bar\delta)\sum r_k$ and $\x^{\phi,\bar\delta} = \x^{\phi,\delta}\circ \lambda$.
Since $|\lambda \circ \bar\lambda| \leq |\lambda| + |\bar\lambda|$, it follows that for every $\delta, \bar \delta > 0$
\[
|\sigma_{\infty}(\x^{\phi,\delta},\bar\x^{\bar\phi,\delta}) - \sigma_{\infty}(\x^{\phi,\bar\delta}, \bar\x^{\bar\phi,\bar\delta})| < 4(\bar \delta + \delta)\sum r_k,
\]
from which the existence of the limit follows.
The fact that~\eqref{eq:limSigma} is independent of the series $\sum r_k$ and is zero if $(\x,\phi)\sim (\bar \x,\bar\phi)$ is straightforward.
\end{proof}

\begin{definition}\label{def:alphaInfty}
Define the pseudometric $\alpha_{\infty}$ on $\DD([0,T],E)$ by
\begin{equation}\label{eq:defAlpha}
\alpha_\infty(\x,\bar \x) := \alpha_{\infty;[0,T]}((\x, \phi), (\bar \x, \bar \phi)) := \lim_{\delta \rightarrow 0} \sigma_{\infty;[0,T]} (\x^{\phi,\delta},\bar \x^{\bar\phi,\delta}).
\end{equation}
(Usually no confusion will arise by using the abusive notation on the left-hand side.)
\end{definition}

Note that for any fixed $\phi$, $\alpha_\infty$ induces a genuine metric on the space $\DD_{\phi}([0,T],E) \subseteq D([0,T],E)$.
We note also that the strong $M_1$ (i.e., $SM_1$) topology on the space $D([0,T],\R^d)$ is a special case of the topology induced by the metric $\alpha_\infty$, as demonstrated by the following result.

\begin{proposition}
For $E=\R^d$ and $\phi$ the linear path function, it holds that $\alpha_\infty$ induces the $SM_1$ topology on the space $D([0,T],\R^d)$.
\end{proposition}
\begin{proof}
It is straightforward to verify that $\alpha_\infty$ in this case is equivalent to the metric $d_s$ (see~\cite[Sec.~12.3.1]{Whitt02}) which induces the $SM_1$ topology.
\end{proof}

\begin{remark}
The reader may wonder if convergence in the (Skorokhod $J_1$) metric $\sigma_\infty$
implies, as in the classical setting, convergence in the (Skorokhod $SM_1$-type) pseudometric $\alpha_\infty$.
In essence, the answer is yes, however this requires ``reasonable'' path functions, see Lemma~\ref{lem:alphSigma}.
\end{remark}

\begin{remark}[Restriction of time interval] \label{remark:alphaRest}
It is trivial to see that uniform convergence of paths on $[0,T]$ implies convergence on any subinterval of $[0,T]$, while this fails for both Skorokhod $J_1$ and $(S)M_1$ metrics.
The observation generalizes to our setting and, in particular, the $\alpha_\infty$ metric does not behave well under restriction.
Indeed, while for any $(\x,\phi) \in \DD([0,T],E)$ the jumps of $\x|_{[s,t]}$ still belong to $J$, so that $(\x|_{[s,t]})^\phi$ is well-defined, it does not hold that $\alpha_{\infty}(\x^n,\x) \rightarrow 0$ implies that $\alpha_{\infty}(\x^n|_{[s,t]},\x|_{[s,t]}) \rightarrow 0$.
\end{remark}

We now collect several useful definitions and lemmas concerning path functions.

\begin{lemma}\label{lem:contPoints}
Let $(\x,\phi) \in \DD([0,T],E)$ for which
\begin{equation}\label{eq:vanishDiagonal}
\lim_{n \rightarrow \infty} \sup_{s \in [0,1]} d(\x_{t_n-}, \phi(\x_{t_n-},\x_{t_n})_s) = 0,
\end{equation}
where the limit is taken over some enumeration of jump times of $\x$. Let $(\x^k,\phi^k)_{k \geq 1}$ be a sequence in $\DD([0,T],E)$ such that $\alpha_{\infty}(\x^k,\x) \rightarrow 0$. Then $\x^k_t \rightarrow \x_t$ for every continuity point of $\x$.
\end{lemma}

\begin{remark}\label{remark:goodCondition}
Note that condition~\eqref{eq:vanishDiagonal} is satisfied whenever $\phi$ is either endpoint continuous or $(\x,\phi)$ has finite $p$-variation (see Definitions~\ref{def:pvar} and~\ref{def:endpointCont} below).
In particular, since c{\`a}dl{\`a}g paths are uniquely determined by their continuity points, it follows from Lemma~\ref{lem:contPoints} that $\alpha_\infty$ is a genuine metric on the space $\DD^{\pvar}([0,T],E)$ introduced in Definition~\ref{def:pvar}.
\end{remark}

\begin{proof}
Suppose $t$ is a continuity point of $\x$.
Then using~\eqref{eq:vanishDiagonal} and the definition of $\alpha_\infty$, it holds that for every $\varepsilon>0$ there exists $\delta> 0$ such that for all $k$ sufficiently large and $\delta_k$ sufficiently small we have
\[
\sup_{s \in [t-\delta,t+\delta]}d(\x_t, (\x^k)^{\phi_k,\delta_k}_s) < \varepsilon,
\]
from which the conclusion follows.
\end{proof}

\begin{definition}\label{def:pvar}
For $p \geq 1$, we define the $p$-variation of $(\x,\phi) \in \DD([0,T],E)$ as
\[
\norm{(\x,\phi)}_{\pvar;[0,T]} := \|\x^\phi\|_{\pvar;[0,T]}
\]
and let $\DD^{\pvar}([0,T],E)$ denote all $(\x,\phi) \in \DD([0,T],E)$ of finite $p$-variation.

Moreover, a path function $\phi : J \to C([0,1], E)$ is called $p$-approximating if there exists a function $\eta_{\pvar} : [0,\infty) \rightarrow [1,\infty)$ such that for all $r \in [0,\infty)$
\[
\sup_{(x,y) \in J; d(x,y) \leq r}\norm{\phi(x,y)}_{\pvar; [0,1]} \leq \eta_{\pvar}(r) d(x,y).
\]
We say that $\eta_{\pvar}$ is a $p$-variation modulus of $\phi$.
\end{definition}

\begin{remark} \label{rmk:welldefined} Due to the invariance of $p$-variation norms under reparametrizations, and our previous Remark \ref{rmk:xPhiwelldefined}, we see that there is no ambiguity in the definition of $\norm{(\x,\phi)}_{\pvar;[0,T]}$ and that $\DD^{\pvar}([0,T],E)$ is well-defined.
\end{remark}

The following lemma gives a simple criterion for a pair $(\x,\phi) \in \DD([0,T],E)$ to have finite $p$-variation.

\begin{lemma}[\cite{Chevyrev15} Lemma~A.5]\label{lem:pvarJumps}
Let $p \geq 1$ and set $R = 1+2^p + 3^{p-1}$. Then for every $(\x,\phi) \in \DD([0,T],E)$, it holds that
\begin{multline*}
\norm{\x}_{\pvar;[0,T]}^p \vee \Big(\sum_{t}\norm{\phi(\x_{t-},\x_{t})}_{\pvar;[0,1]}^p \Big)
\leq \|\x^\phi\|_{\pvar;[0,T]}^p
\\
\leq R\norm{\x}_{\pvar;[0,T]}^p + (R+3^{p-1})\sum_{t}\norm{\phi(\x_{t-},\x_{t})}_{\pvar;[0,1]}^p,
\end{multline*}
where the summations are over the jump times of $\x$.

In particular, if $\phi$ has a $p$-variation modulus $\eta_{\pvar}$, then for all $\x \in \DD_\phi([0,T],E)$,
\[
\|\x^\phi\|_{\pvar;[0,T]}^p
\leq \left[R + \eta_{\pvar}(r)^p(R + 3^{p-1}) \right] \norm{\x}_{\pvar;[0,T]}^p,
\]
where $r = \sup_{t \in [0,T]} d(\x_{t-},\x_t)$.
\end{lemma}

\begin{definition}\label{def:endpointCont}
A path function $\phi : J \to C([0,1], E)$ is called endpoint continuous if
\begin{enumerate}
\item $(x,x) \in J$ whenever $(x,y) \in J$,

\item $\phi(x,x) \equiv x$ for all $(x,x) \in J$, and

\item $\phi$ is continuous with $C([0,1],E)$ equipped with the uniform topology.
\end{enumerate}
Moreover, we say that a function $\eta_\infty : [0,\infty) \rightarrow [0,\infty)$ is a uniform modulus of $\phi$ if $\eta_\infty(r) \geq r$ for all $r \geq 0$, $\lim_{r \rightarrow 0} \eta_\infty(r) = \eta_\infty(0)  = 0$, and for all $(x,y), (\bar x,\bar y) \in J$
\[
d_{\infty;[0,1]}\left(\phi(x,y), \phi(\bar x, \bar y)\right) \leq \eta_\infty(\max\{d(x,\bar x), d(y,\bar y)\}).
\]
\end{definition}

\begin{remark}\label{remark:unifMod}
In general, it is hard to find an explicit uniform modulus of a path function (or even show that one exists). But evidently if $\phi$ is restricted to $J \cap (K \times K)$ for a compact $K \subseteq E$, then a uniform modulus exists whenever $\phi$ is endpoint continuous.
\end{remark}

\begin{example}
Let $\phi$ be the log-linear path function on $G^N(\R^d)$. Then clearly $\phi$ is endpoint continuous and there exists a constant $C \geq 1$ such that for all $p \geq N$ and $x,y \in G^N(\R^d)$
\[
\norm{\phi(x,y)}_{\pvar;[0,1]} \leq C d(x,y),
\]
so that the constant $C$ is a $p$-variation modulus of $\phi$.
\end{example}

\begin{lemma}\label{lem:alphSigma}
Suppose $\phi$ has a uniform modulus $\eta_\infty$. Then for all $\x,\bar\x \in \DD_\phi([0,T],E)$, it holds that $\alpha_{\infty}(\x,\bar\x) \leq \eta_\infty(\sigma_\infty(\x,\bar\x))$.
\end{lemma}

\begin{proof}
Suppose there exists $\lambda\in\Lambda$ such that $|\lambda| < r$ and $d_\infty(\x,\bar\x \circ \lambda) < r$.
Then for all $\delta > 0$ sufficiently small there exists $\lambda_\delta \in \Lambda$ such that $|\lambda_\delta| < r$ and $d_\infty(\x^{\phi,\delta},\bar\x^{\phi,\delta}\circ\lambda_\delta) < \eta_\infty(r)$, and the conclusion follows.
\end{proof}

\section{Canonical RDEs driven by general rough paths}
\label{sec:canonicalRDEs}

To ease notation, we assume throughout this section that all path spaces, unless otherwise stated, are defined on the interval $[0,T]$ and take values in $G^N(\R^d)$. For example $\DD^{\pvar}$ will be shorthand for $\DD^{\pvar}([0,T],G^N(\R^d))$.

\subsection{Notion of solution}
\label{subsec:RDESolution}

Following the notation of Section~\ref{subsec:RPs}, let $1 \leq p < N+1$ and fix a family of vector fields $V = (V_1,\ldots, V_d)$ in $\Lip^{\gamma+m-1}(\R^e)$ for some $\gamma > p$ and $m \geq 1$. For $\x \in D^{\pvar}$, we would like to solve the RDE
\begin{equation*}
\textnormal{``} dy_t = V(y_t)d\x_t \textnormal{''} .
\end{equation*}
Our notion of solution to this equation will depend on a path function $\phi$ defined on a subset $J \subseteq G^N(\R^d) \times G^N(\R^d)$, and therefore the fundamental input to an RDE will be a pair $(\x,\phi) \in \DD^{\pvar}$. 

\begin{definition}[Canonical RDE]\label{def:canonicalRDE}
Consider $(\x,\phi) \in \DD^{\pvar}$ and let $\tilde y \in C^{\pvar}([0,T],\R^e)$ be the solution to the continuous RDE
\[
d\tilde y_t = V(\tilde y_t)d\x_t^\phi, \; \; \tilde y_0 = y_0 \in \R^e.
\]
We define the solution $y \in D^{\pvar}([0,T],\R^e)$ to the canonical RDE
\begin{equation}\label{eq:cadlagRDE}
d y_t = V(y_t)\diamond d(\x_t,\phi), \; \; y_0 \in \R^e,
\end{equation}
by $y = \tilde y \circ \tau_\x$ (where $\tau_\x$ is given by~\eqref{eq:timeChange}).

In the particular case that $\phi$ is the log-linear path function from Example~\ref{ex:logLinear}, we denote the RDE simply by
\[
dy_t = V(y_t) \diamond d\x_t, \; \; y_0 \in \R^e.
\]
\end{definition}

\begin{remark} While the continuous RDE solution $\tilde y$ clearly depends (up to reparametrization) on the choice of representative $(\x,\phi) \in \DD^{\pvar}$ as well as the choice of $(r_k)$, it is easy to see that $y$ is independent of these choices, and is therefore well-defined on $\DD^{\pvar}$.
\end{remark}

\begin{remark}
Observe that every continuity point $t$ of $\x$ is also a continuity point $\tau_\x$, and is therefore also a continuity point of $y$.
\end{remark}

\begin{remark}
For the log-linear path function $\phi$, the solution $y$ agrees precisely with the solution to the rough canonical equation considered in~\cite[Def.~37, Thm.~38]{FrizShekhar17}. Furthermore, we shall see in Section~\ref{sec:cadlagSemimart} that all semimartingales admit a canonical lift to a c{\`a}dl{\`a}g geometric $p$-rough path, and that, for the log-linear path function, the solution $y$ agrees with the Marcus solution of the associated SDE (see Proposition~\ref{prop:Marcus} below).
\end{remark}

\subsection{Skorokhod-type $p$-variation metric}

We now introduce a metric $\alpha_{\pvar}$ on $\DD^{\pvar}$ for which the RDE solution map is locally Lipschitz continuous. We first define an auxiliary metric $\sigma_{\pvar}$ on $D^{\pvar}$ which is independent of any path function. Recall the inhomogeneous $p$-variation metric $\rho_{\pvar}$ from Section~\ref{subsec:RPs}.

\begin{definition}
For $p \geq 1$ and $\x,\bar\x \in D^{\pvar}$, define
\[
\sigma_{\pvar}(\x,\bar\x) = \inf_{\lambda \in \Lambda} \max\{|\lambda|, \rho_{\pvar}(\x \circ \lambda, \bar\x)\}.
\]
\end{definition}

\begin{remark}[Topologies induced by $\sigma_{\pvar}$ and $\rho_{\pvar}$]\label{remark:sigmaRho}
Note that $\sigma_{\onevar}$ and $\rho_{\onevar}$ induce the same topology on $C^{0,\onevar}$. Indeed, it is sufficient to show that $\rho_{\onevar}(\x,\x \circ \lambda^n) \rightarrow 0$ for all $\x \in C^{0,\onevar}$ and $|\lambda^n| \rightarrow 0$, which follows from writing $\x_t = \int_0^t \dot{\x}_s ds$ and applying dominated convergence.

Furthermore, for $p' > p \geq 1$, $\sigma_{\pprimevar}$ and $\rho_{\pprimevar}$ induce the same topology on $C^{\pvar}$. Indeed, it again suffices to show that $\rho_{\pprimevar}(\x,\x \circ \lambda^n) \rightarrow 0$ for all $\x \in C^{\pvar}$ and $|\lambda^n| \rightarrow 0$, which follows from $d_{\infty;[0,T]}(\x, \x\circ \lambda^n) \rightarrow 0$ and interpolation~\cite[Lem.~8.16]{FrizVictoir10}.

However, note that $\sigma_{\onevar}$ and $\rho_{\onevar}$ do not induce the same topology on $C^{\onevar}$. This can be seen from the fact that $C^{0,\onevar}$ is dense in $C^{\onevar}$ under $\sigma_{\onevar}$ (see Proposition~\ref{prop:properties} part~\ref{point:C3}), or from the following direct example: consider the $\R$-valued Cantor function $\x_t = \mu([0,t])$, where $\mu$ is the Cantor distribution, and shifts $\x^n_t = \x_{t-\alpha_n}$ (with $\x^n_t = 0$ for $t \in [0,\alpha_n]$), where $\alpha_n \downarrow 0$. Clearly $\sigma_{\onevar}(\x,\x^n) \leq \alpha_n \vee \norm{\x}_{\onevar;[0,\alpha_n]} \rightarrow 0$. However, choosing $\alpha_n$ irrational, one can show that $\mu$ and $\mu(\cdot - \alpha_n)$ are mutually singular measures (see, e.g.,~\cite{DavisHu95}),
so that $\rho_{\onevar}(\x,\x^n) = \norm{\x^n}_{\onevar;[0,1]} + \norm{\x}_{\onevar;[0,1]} \rightarrow 2$.
\end{remark}

We note that for the case $1 = N \leq p < 2$, the metric $\sigma_{\pvar}$ already appears in the works of Simon~\cite{Simon03} and Williams~\cite{Williams98} where in particular a continuity statement for RDE solutions in the Young regime appears in terms of $\sigma_{\pvar}$ (cf. Remark~\ref{remark:suboptimal}).

A drawback of the metric $\sigma_{\pvar}$ is that the space of continuous rough paths $C^{\pvar}$ is closed under $\sigma_{\pvar}$. In particular, this implies that $\sigma_{\pvar}$ is unable to describe situations in which continuous drivers approximate a discontinuous one (e.g. the Wong--Zakai theorem in~\cite{KPP95}). We are thus motivated to introduce the following metric whose relation with $\sigma_{\pvar}$ is analogous to that of $\alpha_{\infty}$ with $\sigma_{\infty}$.

\begin{definition}
For $(\x,\phi), (\bar\x,\bar\phi) \in \DD^{\pvar}$, define the metric
\begin{equation}\label{eq:defAlphaPVar}
\alpha_{\pvar}(\x,\bar\x) = \lim_{\delta \rightarrow 0} \sigma_{\pvar}(\x^{\phi,\delta},\bar\x^{\bar\phi,\delta}),
\end{equation}
where $\x^{\phi,\delta}$ is defined as at the start of Section~\ref{subsec:SkorSM1}.
\end{definition}

\begin{remark}
Note that the limit~\eqref{eq:defAlphaPVar} exists, is independent of the choice of series $\sum_{k=1}^\infty r_k$, and induces a well-defined metric on $\DD^{\pvar}$, all of which follows from the same argument as in Lemma~\ref{lem:limExists} and Remark~\ref{remark:goodCondition}.
\end{remark}

\begin{remark}
In light of Remark~\ref{remark:sigmaRho}, it may seem possible to define an equivalent topology as that induced by $\alpha_{\pvar}$ (at least on $C^{0,\onevar}$) by replacing $\sigma_{\pvar}$ by $\rho_{\pvar}$ in~\eqref{eq:defAlpha} for the definition of $\alpha_{\pvar}$ (and thus avoid introducing $\sigma_{\pvar}$ altogether). However one can readily check that doing so will induce a completely different topology even on $C^{0,\onevar}$ (in fact the same remark applies to replacing $\sigma_\infty$ by $d_\infty$ when defining $\alpha_\infty$ in Definition~\ref{def:alphaInfty}).
\end{remark}

We record several basic properties of the metric space $(\DD^{\pvar}, \alpha_{\pvar})$. For a path function $\phi$, let $\DD_\phi^{0,\pvar} = \DD_\phi^{0,\pvar}([0,T],G^N(\R^d))$ denote the closure of $C^{0,\onevar}$ in the metric space $(\DD_\phi^{\pvar}, \alpha_{\pvar})$.

\begin{proposition}\label{prop:properties}
Let $p \geq 1$ and $\phi$ a path function defined on a subset $J \subseteq G^N(\R^d)\times G^N(\R^d)$.

\begin{enumerate}[label=\upshape(\roman*\upshape)]
\item \label{point:C1} The space $(\DD^{0,\pvar}_\phi, \alpha_{\pvar})$ is a separable metric space.

\item \label{point:C2} It holds that $C^{\pvar}$ is dense in $(\DD_\phi^{\pvar}, \alpha_{\pvar})$.

\item \label{point:C3} It holds that $\DD_\phi^{0,\onevar} = \DD_\phi^{\onevar}$.

\item \label{point:C4} If $p > 1$, the closure of $C^{0,\onevar}$ in $(C^{\pvar}, \sigma_{\pvar})$ is precisely $C^{0,\pvar}$. In particular, $\DD_\phi^{0,\pvar} \subsetneq \DD_\phi^{\pvar}$.

\item \label{point:C5} For every $p' > p$, $\DD_\phi^{\pvar} \subsetneq \DD_\phi^{0,\pprimevar}$.
\end{enumerate}
\end{proposition}

\begin{proof}
\ref{point:C1} Recall that $(C^{0,\onevar},\rho_{\pvar})$ is a separable space, and therefore so is $(C^{0,\onevar},\sigma_{\pvar})$ (see Remark~\ref{remark:sigmaRho}). Since the metrics $\sigma_{\pvar}$ and $\alpha_{\pvar}$ coincide on $C^{0,\onevar}$, it follows that $\DD_\phi^{0,\pvar}$ is separable.

\ref{point:C2} For every $\x \in \DD_\phi^{\pvar}$ and $\delta > 0$, it holds that $\x^{\phi, \delta} \in C^{\pvar}([0,T],G^N(\R^d))$. One can readily see that $\lim_{\delta \rightarrow 0} \alpha_{\pvar}(\x, \x^{\phi, \delta}) = 0$, from which the claim follows.

\ref{point:C3} By point~\ref{point:C2}, it suffices to show that $C^{0,\onevar}$ is dense in $C^{\onevar}$ under $\sigma_{\onevar}$. This in turn follows from the fact that any $\x \in C^{\onevar}([s,t],G^N(\R^d))$ can be reparametrized to be in $C^{\oneHol}([s,t],G^N(\R^d)) \cong L^\infty([s,t],\R^d)$ (where the isometry is via the weak derivative) and thus lies in $C^{0,\onevar}([s,t],G^N(\R^d)) \cong L^1([s,t],\R^d)$.

\ref{point:C4} Recall by Wiener's characterization~\cite[Thm.~8.22]{FrizVictoir10} (which relies on $p>1$) that $\x \in C^{0,\pvar}$ if and only if 
\[
\lim_{\delta \rightarrow 0} \sup_{|\D|<\delta} \sum_{t_i \in \D} \norm{\x}_{\pvar;[t_i,t_{i+1}]}^p = 0.
\]
As a consequence, for any $\lambda \in \Lambda$, it holds that $\x \circ \lambda \in C^{0,\pvar}$ if and only if $\x \in C^{0,\pvar}$. Therefore, for any sequences $(\x^n)_{n \geq 1}$ in $C^{0,\pvar}$ and $(\lambda^n)_{n \geq 1}$ in $\Lambda$ for which $\rho_{\pvar}(\x,\x^n\circ\lambda_n) \rightarrow 0$, it holds that $\x \in C^{0,\pvar}$, from which the conclusion follows.

\ref{point:C5} Since $C^{\pvar} \subsetneq C^{0,\pprimevar}$, the conclusion follows from~\ref{point:C2}.
\end{proof}

We now record an interpolation estimate which will be helpful later. It turns out to be simpler to state in terms of a homogeneous version of the distance $\alpha_{\pvar}$. See~\cite[Ch.~8]{FrizVictoir10} for the definition and basic properties of the metrics $d_0$ and $d_{\pvar}$. For $(\x,\phi),(\bar\x,\bar\phi) \in \DD^{\pvar}$ we define
\[
\beta_{\pvar}(\x,\bar \x) = \lim_{\delta \rightarrow 0} \inf_{\lambda \in \Lambda} |\lambda| \vee d_{\pvar}(\x^{\phi,\delta} \circ \lambda, \bar\x^{\bar\phi,\delta})
\]
as well as
\[
\alpha_{0}(\x,\bar\x) = \lim_{\delta \rightarrow 0} \inf_{\lambda \in \Lambda} |\lambda| \vee d_{0}(\x^{\phi,\delta} \circ \lambda, \bar\x^{\bar\phi,\delta})
\]
(which are well-defined metrics on $\DD^{\pvar}$ by the same argument as in Lemma~\ref{lem:limExists} and Remark~\ref{remark:goodCondition}).
Observe that the $d_0/d_\infty$ estimate~\cite[Prop.~8.15]{FrizVictoir10}
implies
\begin{multline*}\label{eq:0infty}
\alpha_{\infty}(\x,\bar\x) \leq \alpha_{0}(\x,\bar\x) \leq C\max\big\{ \alpha_{\infty}(\x,\bar\x), \alpha_{\infty}(\x,\bar\x)^{1/N}\big( \|\x^\phi\|_\infty + \|\bar\x^{\bar\phi}\|_\infty \big)^{1-1/N} \big\}.
\end{multline*}
Moreover, to move from $\beta_{\pvar}$ to $\alpha_{\pvar}$, it holds that the identity map
\begin{equation*}\label{eq:alphaBeta}
\id : (\DD^{\pvar}, \beta_{\pvar}) \leftrightarrow (\DD^{\pvar}, \alpha_{\pvar})
\end{equation*}
is Lipschitz on bounded sets in the $\rightarrow$ direction, and $1/N$-H{\"o}lder on bounded sets in the $\leftarrow$ direction~\cite[Thm.~8.10]{FrizVictoir10}.

Finally, the following result now follows directly from the usual interpolation estimate for the homogeneous metric $d_{\pvar}$~\cite[Lem.~8.16]{FrizVictoir10}.

\begin{lemma}\label{lem:interBeta}
Let $1 \leq p < p'$. For all $(\x,\phi),(\bar\x,\bar\phi) \in \DD^{\pvar}$ it holds that
\[
\beta_{\pprimevar}(\x,\bar\x) \leq \big(\|\x^\phi\|_{\pvar} + \|\bar\x^{\bar\phi}\|_{\pvar}\big)^{p'/p} \alpha_{0}(\x,\bar\x)^{1-p'/p}.
\]
\end{lemma}

As a consequence, we obtain the following useful embedding result.

\begin{proposition}\label{prop:SkorkhodAlphaMap}
Let $1 \leq p < p'$ and $\phi$ a $p$-approximating, endpoint continuous path function defined on a subset $J \subset G^N(\R^d) \times G^N(\R^d)$. Then the identity map
\[
\id : (\DD_\phi^\pvar, \sigma_{\infty}) \rightarrow (\DD_\phi^\pvar, \alpha_{\pprimevar})
\]
is uniformly continuous on sets of bounded $p$-variation.
\end{proposition}

\begin{proof}
This is a combination of Remark~\ref{remark:unifMod} and Lemmas~\ref{lem:pvarJumps},~\ref{lem:alphSigma}, and~\ref{lem:interBeta}.
\end{proof}

\subsection{Continuity of the solution map}

An advantage of the metric $\alpha_{\pvar}$ is it allows us to directly carry over continuity statements about the classical (continuous) RDE solution map to the discontinuous setting. Recall the RDE~\eqref{eq:cadlagRDE}
\[
dy_t = V(y_t)\diamond d(\x_t,\phi), \; \; y_0 \in \R^e,
\]
which is well-defined for any admissible pair $(\x,\phi) \in \DD^{\pvar}$.

Consider the driver-solution space $E := G^N(\R^d) \times \R^e$. Given that $\phi$ is defined on $J \subset G^N(\R^d) \times G^N(\R^d)$, we obtain a natural path function $\phi_{(V)}$ defined on the following (necessarily strict)subset of $E \times E$
\[
W := \{ ((x,y),(\bar x,\bar y)) \mid (x,\bar x) \in J, \pi_{(V)}(0,y;\phi(x,\bar x))_1 = \bar y \}
\]
and which is given by
\[
\phi_{(V)}((x,y),(\bar x, \bar y))_t = (\phi(x,\bar x)_t, \pi_{(V)}(0,y;\phi(x,\bar x))_t),
\]
where we recall $\pi_{(V)}$ from Section~\ref{subsec:RPs} is the solution map for the (continuous) RDE driven along $V$.
Therefore, an admissible pair $(\x,\phi) \in \DD^{\pvar}$ yields an admissible pair $((\x,y), \phi_{(V)}) \in \DD^{\pvar}([0,T],E)$ as the solution to the RDE.
The following is now a consequence of Lyons' classical rough path universal limit theorem.

\begin{theorem}[Continuity of solution map]\label{thm:alphaMetric}
For vector fields $V=(V_1,\ldots, V_d)$ in $\Lip^{\gamma+m-1}(\R^e)$ with $\gamma > p$ and $m \geq 1$, the solution map of the RDE~\eqref{eq:cadlagRDE}
\begin{align*}
\R^e \times \left(\DD^{\pvar}, \alpha_{\pvar}\right) &\to \left( \DD^{\pvar}([0,T],E), \alpha_{\pvar} \right) \\
(y_0, (\x,\phi)) & \mapsto ((\x,y),\phi_{(V)})
\end{align*}
is locally Lipschitz. In particular, \begin{equation}\label{eq:ULTdriverConv}
\lim_{n \rightarrow \infty} |y_0^n - y_0| + \alpha_{\pvar}(\x^n,\x) = 0
\end{equation} 
implies that
\begin{align*}
\sup_n \norm{y^n}_{\pvar} < \infty, \; \; \textnormal{and} \; \; \lim_{n \rightarrow \infty} y^n_t = y_t \; \; \textnormal{for all continuity points $t$ of $\x$}.
\end{align*}
Furthermore the flow map
\begin{align*}
\big(\DD_\phi^{\pvar}, \alpha_{\pvar}\big) 
&\rightarrow \Diff^m(\R^e) \\
\x &\mapsto U^{\x}_{T\leftarrow 0} \equiv U^{\x^{\phi}}_{T \leftarrow 0}
\end{align*}
is uniformly continuous on sets of bounded $p$-variation (see Section~\ref{subsec:RPs} for the definition of $U^{\x^{\phi}}_{T \leftarrow 0}$).
\end{theorem}

\begin{remark}
Note that one cannot replace $U^{\x^n}_{T \leftarrow 0}$ by $U^{\x^n}_{t \leftarrow 0}$ for any (fixed) $t \in [0,T]$ in the final statement of Theorem~\ref{thm:alphaMetric}. This is a manifestation of the fact that $\alpha_{\pvar}$ does not behave well under restrictions to subintervals $[0,t] \subset [0,T]$ (cf. Remark~\ref{remark:alphaRest}).
\end{remark}

\begin{remark}
Though we don't address this here, the Lipschitz constant appearing in Theorem~\ref{thm:alphaMetric} can be made to depend explicitly on $V$ and the $p$-variation of $\x^\phi$.
\end{remark}

\begin{remark}
Note that in the second statement of Theorem~\ref{thm:alphaMetric}, one cannot replace $\x$ by $y$ in ``for all continuity points $t$ of $\x$''. Note also that this type of convergence is the one considered in the Wong--Zakai theorem of~\cite{KPP95}.
\end{remark}

\begin{proof}
The claim that the solution map is locally Lipschitz and that the associated flows converge follows from the corresponding result for continuous rough paths (see, e.g.,~\cite[Thm.~10.26]{FrizVictoir10}). 
To make this explicit, consider $\mathbf{x}$ with path function $\phi $ and then $z=\left( \mathbf{x},y\right)$ with path function $\phi _{\left( V\right) }$. Write
also $z^{\delta }=z^{\phi_{\left( V\right) },\delta }$ and $\mathbf{x}
^{\delta }=$ $\mathbf{x}^{\phi ,\delta }$. By definition
\begin{eqnarray*}
\alpha _{p\text{-var}}\left( z_{1},z_{2}\right)  &=&\lim_{\delta \rightarrow
0}\inf_{\lambda \in \Lambda }\max \left\{ \left\vert \lambda \right\vert
,\rho _{p\text{-var}}\left( z_{1}^{\delta }\circ \lambda ,z_{2}^{\delta
}\right) \right\}  \\
&=&\lim_{\delta \rightarrow 0}\inf_{\lambda \in \Lambda }\max \left\{
\left\vert \lambda \right\vert ,\rho _{p\text{-var}}\left( \mathbf{x}
_{1}^{\delta }\circ \lambda ,\mathbf{x}_{2}^{\delta }\right) +\left\vert
y_{1}^{\delta }\circ \lambda -y_{2}^{\delta }\right\vert \right\}.
\end{eqnarray*}
On the other hand, for $i=1,2,$ we have $\mathbf{x}_{i}^{\delta }=\mathbf{x}
_{i}^{\phi _{i},\delta }\in C^{p\text{-var}}$. Write $\tilde{y}_{i}=\pi _{\left(
V\right) }(0,y_{i,0},\mathbf{x}_{i}^{\delta })$ for the unique RDE solution
to $dy=V\left( y\right) d\mathbf{x}_{i}^{\delta }$ and note that $\tilde{y}
_{i}\equiv y_{i}^{\delta }$, by the very definition of the path function $\phi _{(V)}$.\ Note also that 
\begin{equation*}
y_{i}^{\delta }\circ \lambda \equiv \tilde{y}^{i}\circ \lambda =\pi _{\left(
V\right) }(0,y_{0}^{i},\mathbf{x}_{i}^{\delta }\circ \lambda )
\end{equation*}
for every time change $\lambda \in \Lambda $. It is then a direct
consequence of the (local Lipschitz) continuity of the It{\^o}-Lyons map, in the
setting of continuous rough paths, that 
\begin{equation*}
\left\vert y_{1}^{\delta }\circ \lambda -y_{2}^{\delta }\right\vert _{p\text{-var;}\left[ 0,T\right] }\lesssim \rho _{p\text{-var}}\left( \mathbf{x}
_{1}^{\delta }\circ \lambda ,\mathbf{x}_{2}^{\delta }\right) +\left\vert
y_{1,0}-y_{2,0}\right\vert.
\end{equation*}
As a consequence,
\begin{equation*}
\rho _{p\text{-var}}\left( z_{1}^{\delta }\circ \lambda ,z_{2}^{\delta
}\right) \lesssim \rho _{p\text{-var}}\left( \mathbf{x}_{1}^{\delta }\circ
\lambda ,\mathbf{x}_{2}^{\delta }\right) +\left\vert
y_{1,0}-y_{2,0}\right\vert.
\end{equation*}
Finally, take $\lim_{\delta \rightarrow 0}\inf_{\lambda \in \Lambda }\max
\left( \left\vert \lambda \right\vert ,\cdot \right) $ on both sides to see
that
\begin{equation*}
\alpha _{p\text{-var}}\left( z_{1},z_{2}\right) \lesssim \alpha _{p\text{-var}}\left( \mathbf{x}_{1},\mathbf{x}_{2}\right) +\left\vert
y_{1,0}-y_{2,0}\right\vert .
\end{equation*}
The claim of a.e. pointwise convergence follows from Lemma~\ref{lem:contPoints}, while uniform continuity of $(\x,\phi) \mapsto U^{\x}_{T \leftarrow 0}$ follows as above (cf.~\cite[Thm.~11.12]{FrizVictoir10}).
\end{proof}

\begin{remark}
An important feature of solutions to continuous RDEs~\eqref{eq:cadlagRDE} is that they can be canonically treated as geometric $p$-rough paths $\y \in C^{\pvar}([0,T], G^N(\R^e))$, and the solution map $(\y_0,\x) \mapsto \y$ remains locally Lipschitz for the metric $\rho_{\pvar}$ (see~\cite[Sec.~10.4]{FrizVictoir10}). This allows one to use the solution $\y$ as the driving signal in a secondary RDE
\begin{equation}\label{eq:zRDE}
dz_t = W(z_t) d\y_t, \; \; z_0 \in \R^{n}.
\end{equation}
At least with the notion of canonical RDEs considered in this article, we cannot expect this functorial nature to be completely preserved due to the fact that in general no path function $\psi$ can be defined on $G^N(\R^e)$ to capture the information of how $\y$ traversed a jump $(\y_{t-},\y_t)$ (which, in our context, clearly impacts	the solution of~\eqref{eq:zRDE}).

However, at the cost of retaining the original driving signal $\x$, we can readily solve the secondary RDE~\eqref{eq:zRDE}. Indeed, for $\Lip^\gamma$ vector fields $V=(V_1,\ldots, V_d)$ on $\R^e$ and $W=(W_1,\ldots, W_e)$ on $\R^n$, we consider the $\Lip^\gamma$ vector fields $U=(U_1,\ldots,U_d)$ on $\R^e\oplus \R^n$ defined by $U_i(y,z) = V_i(y) + \sum_{k=1}^e W_k(z) V_i(y)^k$, where $V_i(y)^k$ is the $k$-th coordinate of $V_i(y) \in \R^e$. Then the natural solution to~\eqref{eq:zRDE} is given by the larger RDE
\begin{equation}\label{eq:yxRDE}
d(y_t,z_t) = U(y_t,z_t) \diamond d(\x_t,\phi), \; \; (y_0,z_0) \in \R^e \oplus \R^n.
\end{equation}
As a consistency check, one can readily see that, in the continuous setting, the second component of the solution to~\eqref{eq:yxRDE} coincides with the solution $z_t$ of~\eqref{eq:zRDE}.
\end{remark}

If we further assume that drivers converge in Skorokhod topology, then more can be said about convergence of RDE solutions.

\begin{proposition}\label{prop:alphaSkorokhod}
Let notation be as in Theorem~\ref{thm:alphaMetric}.

\begin{enumerate}[label=\upshape(\roman*\upshape)]
\item \label{point:contSkor1} Suppose that~\eqref{eq:ULTdriverConv} holds and that $\lim_{n \rightarrow \infty} \x^n = \x$ in the Skorokhod topology. Then $\lim_{n \rightarrow \infty} y^n = y$ in the Skorokhod topology.

\item \label{point:contSkor2} Suppose that $\phi$ is an endpoint continuous, $p$-approximating path function defined on (a subset of) $G^N(\R^d) \times G^N(\R^d)$. Then on sets of bounded $p$-variation, the solution map
\begin{align*}
\R^e \times \big(\DD_\phi^{\pvar}, \sigma_{\infty}\big) 
&\rightarrow \left( D^{\pvar}([0,T],\R^e), \sigma_{\infty} \right) \\
(y_0, \x) &\mapsto y
\end{align*}
is continuous.
\end{enumerate}
\end{proposition}

\begin{proof}
\ref{point:contSkor1} By Theorem~\ref{thm:alphaMetric}, it suffices to show that $(y^n)_{n \geq 1}$ is compact in the Skorokhod space $D([0,T],\R^e)$.
Recall that for a Polish space $E$, a subset $A\subset D([0,T],E)$ is compact if and only if $\{\y_t \mid \y\in A, t \in [0,T]\}$ is compact and
\[
\lim_{\varepsilon\to 0} \sup_{\y \in A} \omega'_{\y}(\varepsilon) = 0,
\]
where 
\[
\omega'_{\y}(\varepsilon) := \inf_{|\D|_{\min} > \varepsilon} \max_{t_i \in \D} \sup_{s,t \in [t_i,t_{i+1})} d(\y_s,\y_t), \; \; |\D|_{\min} := \min_{t_i \in \D}|t_{i+1} - t_i|,
\]
see, e.g.,~\cite[Thm.~12.3]{Billingsley99}.
In particular, using that $\|\x^\phi\|_{\pvar}< \infty$ and applying the first inequality in Lemma~\ref{lem:pvarJumps} to $\x$, we see that $\lim_{\varepsilon \to 0}\lim_{\delta \to 0}\omega^{\pvar}_{\x^{\phi,\delta}}(\varepsilon)=0$, where
\[
\omega^{\pvar}_{\y}(\varepsilon) := \inf_{|\D|_{\min} > \varepsilon} \max_{t_i \in \D} \norm{\y}_{\pvar;[t_i,t_{i+1})}.
\]
It now follows from~\eqref{eq:ULTdriverConv} that
\[
\lim_{\varepsilon \rightarrow 0} \sup_n \lim_{\delta \rightarrow 0} \omega^{\pvar}_{(\x^n)^{\phi^n,\delta}}(\varepsilon) = 0,
\]
from which we see that $\lim_{\varepsilon \rightarrow 0} \sup_n \omega^{\pvar}_{y^n}(\varepsilon) = 0$. Since $\omega'_y(\varepsilon) \leq \omega^{\pvar}_y(\varepsilon)$, it follows that $(y^n)_{n \geq 1}$ is compact in $D([0,T],\R^e)$ as required.

\ref{point:contSkor2} This follows directly from taking $p < p' < \gamma$ and applying~\ref{point:contSkor1} and Proposition~\ref{prop:SkorkhodAlphaMap}.
\end{proof}

\begin{remark}\label{remark:suboptimal}
We suspect that under suitable conditions on a fixed path function $\phi$, the solution map $\DD_\phi^{\pvar} \rightarrow D^{\pvar}([0,T],\R^e)$ remains locally uniformly (or even Lipschitz) continuous in the classical $p$-variation metric $\rho_{\pvar}$ defined by~\eqref{eq:rhoDef} (and thus trivially for the metric $\sigma_{\pvar}$). However there seems to be no easy way to derive this as a consequence of our main Theorem~\ref{thm:alphaMetric}.
We suspect that this could be done by carefully relating canonical with ``It{\^o}-type'' \ {\it non-canonical} RDEs (see~\cite[Def.~37]{FrizShekhar17}), followed by a proper stability analysis of the latter.
The study of such non-canonical equations was recently carried out in~\cite{FZ17}.
\end{remark}

\subsection{Young pairing and translation operator}

In this section we extend the Young pairing $S_{N}(\x,h)$ and rough path translation operator $T_h(\x)$ to the c{\`a}dl{\`a}g setting.

Given a path $(x,h)$ in $\R^{d+d'}$ (smooth), its canonical level-$2$ rough path lift is given by
$((x,h), (\int x \otimes dx, \int x \otimes dh, \int h \otimes dx, \int h\otimes dh))$.
This extends immediately to a $p$-rough path $\x$, for $p \in (2,3)$, with $\int x \otimes dx$ replaced by the a priori level-$2$ information $\x^2$.
For $h$ of finite $q$-variation with $1/p+1/q > 1$, all other cross-integrals remain defined. This is the Young pairing of a $p$-rough path $\x$ with a $q$-variation path $h$ called $S_2 (\x, h)$.
The mapping $(\x,h) \mapsto S_2(\x,h)$ is continuous, and the general construction, for continuous (rough) paths, is found in~\cite[Sec.~9.4]{FrizVictoir10}.
The important operation of translating rough paths in some $h$-direction, formally $\int (\x+h) \otimes d(\x+h)$, can be algebraically formulated in terms of the Young pairing and has found many applications~\cite{FrizHairer14}.

We first illustrate the difficulty of \emph{continuously} extending $S_{N}(\x,h)$ and $T_h(\x)$ to c{\`a}dl{\`a}g paths by showing that the addition map $(\x,h) \mapsto \x+h$ is not continuous as a map $\DD^{\onevar}([0,T],\R^d) \times \DD^{\onevar}([0,T],\R^d) \rightarrow \DD([0,T],\R^d)$ where we equip $\DD^{\onevar}([0,T],\R^d)$ and $\DD([0,T],\R^d)$ with $\alpha_{\onevar}$ and $\alpha_{\infty}$ respectively (which is reminiscent of the fact that $D([0,T],\R^d)$ is not a topological vector space).

\begin{example}
Consider the sequences $(\x^n)_{n \geq 1}$, $(h^n)_{n \geq 1}$, and $(\bar h^n)_{n \geq 1}$, of continuous piecewise linear paths from $[0,2]$ to $\R^2$, where $\x^n$ is constant on $[0,1-1/n]$, moves linearly from $0$ to $e_1$ over $[1-1/n,1]$ and remains constant on $[1,2]$, while $h^n$ and $\bar h^n$ do the same, except from $0$ to $e_2$ and over the intervals $[1-2/n,1-1/n]$ and $[1,1+1/n]$ respectively. 

One can see that $\x^n$ converges to $(\1{t \geq 1} e_1, \phi)$ in $\alpha_{\onevar}$, and $h^n$ and $\bar h^n$ both converge to $(\1{t \geq 1} e_2, \phi)$ in $\alpha_{\onevar}$, where $\phi$ is the linear path function on $\R^2$. However $\x^n + h^n$ converges to $(\1{t \geq 1}(e_1+e_2), \phi_{2,1})$ while $\x^n + \bar h^n$ converges to $(\1{t \geq 1}(e_1+e_2), \phi_{1,2})$ in $\alpha_{\infty}$, where $\phi_{i,j}$ is the ``Hoff'' path function which moves first in the $i$-th coordinate, and then in the $j$-th coordinate.
\end{example}

Note that the limiting path in the above example is unambiguously defined (namely $\1{t \geq 1}(e_1+e_2)$) whereas
the corresponding path function is not. In the following definition we circumvent this 
problem by choosing a priori the (log-)linear path function.

\begin{definition}\label{def:translation}
Let $p \geq 1$ and $1 \leq q < 2$ such that $1/p + 1/q > 1$.
For integers $d,d' \geq 1$, let $\phi$ and $\bar\phi$ be the log-linear path functions on $G^N(\R^{d+d'})$ and $G^N(\R^{d})\times \R^{d'}$ respectively.

For $\x \in D^{\pvar}([0,T],G^N(\R^d))$ and $h \in D^{\qvar}([0,T],\R^{d'})$, consider the continuous $G^N(\R^d)\times \R^{d'}$-valued path $(\hat{\x},\hat h) = (\x,h)^{\bar\phi}$.
We define $(S_N(\x,h),\phi) \in \DD^{\pvar}([0,T],G^N(\R^{d+d'}))$ by
\begin{equation}\label{eq:translationDef}
S_N(\x,h) = S_N(\hat{\x},\hat h) \circ \tau_{(\x,h)},
\end{equation}
where we recall that $\tau_{(\x,h)}$ is defined by~\eqref{eq:timeChange} for the c{\`a}dl{\`a}g path $(\x,h)$.
In the case that $d=d'$, we let $\hat \phi$ denote the log-linear path function on $G^N(\R^d)$ and define $(T_h(\x),\hat\phi) \in \DD^{\pvar}([0,T],G^N(\R^{d}))$ by $T_h(\x) = T_{\hat h}(\hat{\x}) \circ \tau_{(\x,h)}$.
\end{definition}

\begin{remark}
Due to the choice of log-linear path function, note that $\hat{\x}$ and $\hat h$ have finite $p$- and $q$-variation respectively, so that $S_N(\hat{\x},\hat h)$ and $T_{\hat h}(\hat{\x})$ are well-defined as continuous rough paths.
Moreover, we have the relation $T_h(\x) = \plus\left( S_N(\x,h)\right)$ (see~\cite[Thm.~9.33]{FrizVictoir10}).
\end{remark}

We now record a simple result in the case $N=2$ which will be helpful in the remainder of the paper.
Recall from Definition~\ref{def:RPs} that $\x \in D^{\pvar}([0,T],G^2(\R^d))$ is called Marcus-like if $\log(\Delta\x_t) \in \R^d\oplus\{0\} \subset \g^2(\R^d)$, i.e., $\x = \exp(x, A)$ where $\exp$ is taken in $T^2(\R^d)$ and $A_{t-,t} = 0$ for all $t \in [0,T]$.
\begin{proposition}\label{prop:translation}
Let notation be as in Definition~\ref{def:translation}.
Suppose that $\x \in D^{\pvar}([0,T],G^2(\R^d))$ is Marcus-like.
Then for every $h \in D^{\qvar}([0,T],\R^{d'})$ it holds that $S_2(\x,h)$ is a weakly geometric Marcus-like $p$-rough path with the anti-symmetric part of its second level determined for $i\in\{1,\ldots, d\}$ and $i' \in \{1\ldots, d'\}$ by
\begin{equation}\label{eq:translationIntegrals}
S_2(\x,h)^{i,d+i'}_{s,t} - S_2(\x,h)^{d+i',i}_{s,t} = \int_{(s,t]} x^{i}_{s,u-} \otimes dh^{i'}_u - h_{s,u-} ^{i'}\otimes dx^i_u,
\end{equation}
where the integrals are well-defined Young integrals (all other components of the anti-symmetric part are given canonically by $\x$ and $h$).
Furthermore
\begin{equation}\label{eq:translationInequality}
\|S_2(\x,h)^\phi\|_{\pvar}\vee \|(T_h(\x)^{\hat\phi}\|_{\pvar} \lesssim \|(\x,h)^{\bar\phi}\|_{\pvar} \lesssim \norm{\x}_{\pvar} + \norm{h}_{\qvar}.
\end{equation}
\end{proposition}

\begin{proof}
By definition of $\ZZ := (Z,\Z) := S_2(\x,h)$ and the choice of log-linear path functions, we see that $\ZZ$ is indeed Marcus-like.
Observe that $\tilde{\ZZ} := S_2(\hat{\x},\hat h)$ satisfies the linear RDE in $T^N(\R^{d+d'})$ driven by itself $d\tilde{\ZZ}_t = \tilde{\ZZ}_t\otimes d\tilde{\ZZ}$.
It follows by~\cite[Thm.~38]{FrizShekhar17} and the definition of $\tau_{(\x,h)}$ that $\ZZ$ satisfies the canonical rough equation $d\ZZ_t = \ZZ_t\otimes \diamond d\ZZ_t$ (in the sense of~\cite[Def.~37]{FrizShekhar17}), so that
\begin{align*}
\ZZ_t &= \ZZ_0 + \int_0^t \ZZ_{s-}\otimes d\ZZ_s + \sum_{0<s\leq t} \ZZ_{s-} \otimes \exp(\Delta Z_s) - \ZZ_{s-} - \ZZ_{s-}\otimes \Delta Z_s
\\
&= 1_{G^N(\R^{d+d'})} + \int_0^t \ZZ_{s-}\otimes d\ZZ_s + \sum_{0<s\leq t} \frac{1}{2}(\Delta Z_s)^{\otimes 2},
\end{align*}
where we have used that $\ZZ$ is Marcus-like.
In particular, the last term on the RHS only contributes to the symmetric part of $\Z$, so that~\eqref{eq:translationIntegrals} follows by identifying $\int_0^t\ZZ_{s-}\otimes d\ZZ_s$ in the appropriate components with well-defined Young integrals.
The inequality~\eqref{eq:translationInequality} follows from standard (continuous) rough path estimates~\cite[Thm.~9.26,~9.33]{FrizVictoir10} along with the fact that $\x$ is Marcus-like (so that upon appropriately restricting domains, we may assume $\phi$, $\bar \phi$, and $\hat\phi$ are $1$-approximating).
\end{proof}

\section{General multidimensional semimartingales as rough paths}
\label{sec:cadlagSemimart}

\subsection{Enhanced $p$-variation BDG inequality}\label{subsec:BDG}

The main result of this subsection is the BDG inequality Theorem~\ref{thm:pvarBDG} for enhanced c{\`a}dl{\`a}g local martingales. The proof largely follows a classical argument found in~\cite{Lepingle76, FrizVictoir10} with the exception of Lemma~\ref{lem:interpolation} which constitutes our main novel input.

Let $X$ be an $\R^d$-valued semimartingale with $X_0 = 0$ and $(\R^d)^{\otimes 2}$-valued bracket $[X]$. Consider the $\g^{(2)}(\R^d) = \mathfrak{so}(\R^d)$-valued (``area'') process
\begin{equation}\label{eq:stochArea}
A_{t}^{i,j} = \frac{1}{2}\int_0^t X_{r-}^i dX^j_r - X^j_{r-} dX^i_{r}
\end{equation}
as an It{\^o} integral, and the ``Marcus lift'' (terminology from \cite{FrizShekhar17}) 
of $X$ to a $G^2(\R^d)$-valued process given by
\[
\X_t := \exp(X_t + A_t),
\]
where $\exp$ is taken in the truncated tensor algebra $T^2(\R^d)$.

Note that, if $X = Y + K$ where $Y$ is (another) semimartingale (w.r.t. the same filtration) and $K$ is adapted, of bounded variation, then the respective Marcus lifts of $X$ and $Y$ are precisely related by the translation operator; that is, $\X = T_K \Y$.
This can be seen by combining~\eqref{eq:translationIntegrals} with the fact that the second level of $\X$  is given precisely by the Marcus canonical integral $\int X\otimes \diamond dX$; see~\cite[Prop.~16]{FrizShekhar17} for details.

Recall that a function $F : [0,\infty) \to [0,\infty)$ is called \emph{moderate} if it is continuous and increasing, $F(x)=0$ if and only if $x=0$, and there exists $c>0$ such that $F(2x) \leq cF(x)$ for all $x > 0$.

\begin{lemma}[Uniform enhanced BDG inequality]\label{lem:unifBDG}
For any convex moderate function $F$ there exist $c,C > 0$ such that for any $\R^d$-valued local martingale $X$
\[
c \EEE{F \left(\norms{[X]_\infty} \right)} \leq \E\Big[\sup_{s,t \geq 0} F \left(\norm{\X_{s,t}}^2\right)\Big] \leq C \EEE{F \left(\norms{[X]_\infty} \right)}.
\]
\end{lemma}

\begin{proof}
For a process $M$, denote $M_t^* := \sup_{0 \leq s \leq t}|M_s|$. Following the proof of~\cite[Thm.~14.8]{FrizVictoir10}, it suffices to show that
\[
\EEE{F (|A_\infty^*|)} \leq c_2 \EEE{F (\norms{[X]_\infty})}. 
\]
However we have
\[
\norms{[A]_\infty}^{1/2} \leq c_1|X_\infty^*|\norms{[X]_\infty}^{1/2} \leq c_1(|X_\infty^*|^2 + \norms{[X]_\infty}),
\]
so one can apply the classical BDG inequality (e.g.,~\cite[Thm.~2.1]{Lenglart80}) to $A$ and $F$, and to $X$ and $F(|\cdot|^2)$, which is also a convex moderate function, to obtain 
\[
\EEE{F(|A^*_\infty|)} \leq c_2\EEE{F\left(\norms{[A]_\infty}^{1/2}\right)} \leq c_3\EEE{F(|X^*_\infty|^2) + F(\norms{[X]_\infty})} \leq c_4\EEE{F(\norms{[X]_\infty})}.
\]
\end{proof}

The following lemma is the crucial step in establishing finite $p$-variation of the lift of a local martingale.

\begin{lemma}[Interpolation]\label{lem:interpolation}
For every $2 < q < p < r$ there exists $C = C(p,q,r)$ such that for every $\R^d$-valued local martingale $X$
\[
\E\big[\norm{\X}_{\pvar}^p\big] \leq C\E\big[\norms{[X]_\infty}^{q/2} + \norms{[X]_\infty}^{r/2}\big].
\]
\end{lemma}

\begin{proof}
For $\delta > 0$ define the increasing sequence of stopping times $(\tau^\delta_j)_{j=0}^\infty$ by $\tau^\delta_0 = 0$ and for $j \geq 1$
\[
\tau^\delta_j = \inf \big\{t \geq \tau^\delta_{j-1} \mid \sup_{u,v \in [\tau^b_{j-1},t]} d(\X_u,\X_v) > \delta\big\}
\]
(where $\inf$ of the empty set is $\infty$). Define further $\nu(\delta) := \inf\{j \geq 0 \mid \tau^\delta_j = \infty\}-1$. Observe that (cf.~\cite[p.~10]{Chevyrev15})
\begin{equation}\label{eq:pVarBoundNuM}
\norm{\X}^p_{\pvar} \leq \sum_{k = -\infty}^\infty 2^{p(k+1)}\nu(2^k).
\end{equation}

Fix $\delta > 0$ and denote $\tau_j := \tau^\delta_j$. For $j=0,1,\ldots$ consider the sequence of local martingales $Y^j_t := X_{(\tau_j+t)\wedge \tau_{j+1}} - X_{\tau_j}$. Denote by $\Y^j_t$ the lift of $Y^j_t$, which coincides with $\X_{\tau_j, (\tau_j + t)\wedge \tau_{j+1}}$.

It holds that
\[
\sum_{j=0}^\infty [Y^j]_\infty = \sum_{j=0}^{\nu(\delta)} [Y^j]_\infty = [X]_\infty,
\]
and moreover $\sup_{s,t \geq 0} \|\Y^j_{s,t}\| \geq \delta$ for all $j = 0, \ldots, \nu(\delta)-1$. Thus by the uniform enhanced BDG inequality (Lemma~\ref{lem:unifBDG}) with $F = |\cdot|^{\alpha}$, $\alpha \geq 1$,
\begin{equation*}
\EEE{\norms{[X]_\infty}^\alpha} \geq \E\Big[\sum_{j=0}^\infty \norms{[Y^j]_\infty}^\alpha\Big] \geq c_\alpha \E\Big[\sum_{j=0}^\infty \sup_{s,t \geq 0} \|\Y^j_{s,t}\|^{2\alpha}\Big] \geq c_\alpha\delta^{2\alpha}\EEE{\nu(\delta)}.
\end{equation*}
It follows that for $2 \leq 2\alpha < p$
\[
\E\Big[\sum_{k \leq 0} 2^{p(k+1)}\nu(2^k)\Big] \leq c_\alpha^{-1}\EEE{\norms{[X]_\infty}^\alpha} \sum_{k \leq 0} 2^{p(k+1)-2\alpha k} \leq C(\alpha, p) \EEE{\norms{[X]_\infty}^\alpha},
\]
and likewise for $2\beta > p$
\[
\E\Big[\sum_{k > 0} 2^{p(k+1)}\nu(2^k)\Big] \leq C(\beta, p) \E\big[\norms{[X]_\infty}^\beta\big].
\]
Taking $2\alpha = q$ and $2\beta = r$, the conclusion follows from the estimate~\eqref{eq:pVarBoundNuM}.
\end{proof}

\begin{corollary}\label{cor:semimartPvar}
For every $\R^d$-valued semimartingale $X$, $p > 2$, and $T > 0$, it holds that a.s.
\[
\norm{\X}_{\pvar;[0,T]} < \infty.
\]
\end{corollary}

\begin{proof}
Note that $X$ can be decomposed into a process $K$ of finite variation and a local martingale $L$ with jump sizes bounded by a positive constant (see, e.g.,~\cite[Prop.~4.17, p.~42]{JacodShiryaev03}). Denoting by $\LLL$ the lift of $L$, it follows from a localization argument and Lemma~\ref{lem:interpolation}, that $\norm{\LLL}_{\pvar;[0,T]} < \infty$ a.s.. The conclusion follows by observing that $\X = T_K(\LLL)$.
\end{proof}

\begin{lemma}[Chebyshev inequality]\label{lem:Chebyshev}
For all $p > 2$, there exists a constant $A > 0$ such that for every $\R^d$-valued local martingale $X$ and $\lambda > 0$,
\[
\PP\big[\norm{\X}_{\pvar} > \lambda\big] \leq \frac{A}{\lambda^2}\EEE{\norms{[X]_\infty}}.
\]
\end{lemma}

\begin{proof}
This crucially uses that $\X$ has finite $p$-variation for every local martingale $X$, and follows in exactly the same manner as~\cite[Lem.~1]{Lepingle76} or~\cite[Lem.~14.10]{FrizVictoir10}.
\end{proof}

\begin{lemma}[Burkholder~\cite{Burkholder73} Lemma~7.1] \label{lem:Burkholder}
Suppose $X$ and $Y$ are non-negative random variables, $F$ is a moderate function, and $\beta > 1$ and $\delta, \varepsilon, \gamma, \eta > 0$ such that $\gamma\varepsilon < 1$,
\[
F(\beta\lambda) \leq \gamma F(\lambda), \; \; F(\delta^{-1}\lambda) \leq \eta F(\lambda), \; \; \forall \lambda > 0,
\]
and
\[
\PPP{X > \beta\lambda, Y < \delta\lambda} \leq \varepsilon \PPP{X > \lambda}, \; \; \forall \lambda > 0.
\]
Then
\[
\EEE{F(X)} \leq \frac{\gamma\eta}{1-\gamma\varepsilon} \EEE{F(Y)}.
\]
\end{lemma}

\begin{lemma}\label{lem:previsBound}
Let $X$ be an $\R^d$-valued local martingale and $D$ an adapted, non-decreasing process such that a.s., $|\Delta X_t| \leq D_{t-}$ for all $t \geq 0$. Then for every moderate function $F$ (not necessarily convex), there exists $C = C(F) > 0$ such that
\[
\E\big[F(\norm{\X}_{\pvar})\big] \leq C\E\big[F(\norms{[X]_\infty}^{1/2} + D_\infty)\big].
\]
\end{lemma}

\begin{proof}
We follow closely the proof of~\cite[Prop.~2]{Lepingle76} and~\cite[Thm.~14.12]{FrizVictoir10}.
Since $\X$ is Marcus-like, i.e., $\log(\Delta\X_t) \in \R^d$, there exists a constant $c>0$ such that for all $t \geq 0$
\begin{equation}\label{eq:jumpBound}
\norm{\Delta\X_t} = c|\Delta X_t| \leq c\norms{[X]_t}^{1/2}.
\end{equation}
Let $\delta > 0$, $\beta > c\delta + 1$, $\lambda > 0$, and define the stopping times
\begin{align*}
T &= \inf \big\{ t \geq 0 \mid \norm{\X}_{\pvar;[0,t]} > \beta\lambda \big\}, \\
S &= \inf \big\{ t \geq 0 \mid \norm{\X}_{\pvar;[0,t]} > \lambda \big\}, \\
R &= \inf \big\{ t \geq 0 \mid D_t \vee \norms{[X]_t}^{1/2} > \delta\lambda \big\}.
\end{align*}
Define the local martingale $N_t = X_{(t + S)\wedge R} - X_{S\wedge R}$ with lift $\Nbf$ and note that
\[
\norm{\Nbf}_{\pvar} \geq \norm{\X}_{\pvar;[0,R]} - \norm{\X}_{\pvar;[0,R\wedge S]}.
\]
On the event $\{T < \infty, R = \infty\}$, we have $\Delta X_S \leq D_{S-} \leq \delta\lambda$, and so from~\eqref{eq:jumpBound}
\[
\norm{\Nbf}_{\pvar} \geq \beta \lambda - \lambda - \norm{\Delta \X_S} \geq (\beta - 1 - c\delta)\lambda.
\]
By Lemma~\ref{lem:Chebyshev}, it follows that
\[
\PPP{T < \infty, R = \infty} \leq \PP\big[\norm{\Nbf}_{\pvar} > (\beta - c\delta - 1)\lambda\big] \leq \frac{A}{(\beta - c\delta - 1)^2\lambda^2} \EEE{\norms{[N]_\infty}}.
\]
On the event $\{S = \infty\}$, it holds that $N \equiv 0$, whilst on $\{S < \infty\}$, we have $D_{R-} \leq \delta\lambda$ and thus
\[
\norms{[N]_\infty} = \norms{[X]_R - [X]_{R\wedge S}} \leq \norms{[X]_{R-}} + |\Delta X_R|^2 \leq 2\delta^2\lambda^2.
\]
It follows that
\[
\EEE{\norms{[N]_\infty}} \leq 2\delta^2\lambda^2\PPP{S < \infty},
\]
and thus we have for all $\lambda > 0$
\[
\PP\big[\norm{\X}_{\pvar} > \beta\lambda, D_\infty\vee [X]_\infty^{1/2} \leq \delta\lambda\big] \leq \frac{2A\delta^2}{(\beta - c\delta - 1)^2}\PP\big[\norm{\X}_{\pvar} > \lambda\big].
\]
The conclusion now follows by applying Lemma~\ref{lem:Burkholder}.
\end{proof}

\begin{theorem}[$p$-variation rough path BDG]\label{thm:pvarBDG}
For every convex moderate function $F$ and $p > 2$ there exists $c, C > 0$ such that for every $\R^d$-valued local martingale $X$
\[
c\E\big[F\big(\norms{[X]_\infty}^{1/2}\big)\big] \leq \E\big[F \big(\norm{\X}_{\pvar} \big)\big] \leq C\E\big[F \big(\norms{[X]_\infty}^{1/2} \big)\big].
\]
\end{theorem}

\begin{proof}
This again follows very closely the proof of~\cite[Prop.~2]{Lepingle76}.
We may suppose $\E[\norms{[X]_\infty}^{1/2}] < \infty$ (otherwise all concerned quantities are infinite). 
Let $D_t := \sup_{0 \leq s \leq t} |\Delta X_t|$ and $K^1_t := \sum_{s \leq t} \Delta X_s \1{|\Delta X_s| \geq 2D_{s-}}$.
Note that $K^1$ is of integrable variation since
\[
|K^1|_{\onevar} \leq 4D_\infty \leq 4\norms{[X]_\infty}^{1/2},
\]
so there exists a unique previsible process $K^2$ such that $K^1-K^2$ is a martingale ($K^2$ is the dual previsible projection of $K^1$~\cite[Thm.~21.4]{Rogers00} and is a special case of the Doob--Meyer decomposition).
Finally, define the martingale $L := X - (K^1-K^2)$; recall that $X = (K^1 - K^2) + L$ is called the Davis decomposition of $X$.

Observe that $|\Delta L_t| \leq 4D_{t-}$. Indeed, $|\Delta(L-K^2)_t| \leq 2D_{t-}$ by construction, and, if $T$ is a stopping time, then either $T$ is totally inaccessible, in which case $\Delta K^2_T = 0$ since $K^2$ is previsible, or $T$ is previsible, in which case $|\Delta K^2_T| = |\EEE{\Delta (K^2-L)_T\mid \FF_{T-}}| \leq \EEE{2D_{T-}\mid \FF_{T-}}=2D_{T-}$, where $\FF$ is the underlying filtration (cf.~\cite[p.~80-81]{Meyer72}).
Hence, by Lemma~\ref{lem:previsBound}, we have
\[
\E\big[F\big(\norm{\LLL}_{\pvar}\big)\big] \leq C_1 \E\big[F\big(\norms{[L]_\infty}^{1/2} + D_\infty\big)\big].
\]
Since $D_\infty \leq \norms{[X]_\infty}^{1/2}$ and
\[
\norms{[L]_\infty}^{1/2} \leq \norms{[X]_\infty}^{1/2} + \norms{[K]_\infty}^{1/2} \leq \norms{[X]_\infty}^{1/2} + |K|_{\onevar},
\]
we have 
\[
\E\big[F\big(\norm{\LLL}_{\pvar}\big)\big] \leq C_2 \E\big[F\big(\norms{[X]_\infty}^{1/2} + |K|_{\onevar}\big)\big].
\]
Furthermore, since $F$ is convex, it follows from the Garsia-Neveu lemma (using the argument provided by~\cite[p.~306]{Lepingle76}) that
\[
\EEE{F\left(|K^2|_{\onevar}\right)} \leq C_3\EEE{F\left(|K^1|_{\onevar}\right)},
\]
and thus
\[
\EEE{F\left(|K|_{\onevar}\right)} \leq C_4\E\big[F\big(\norms{[X]_\infty}^{1/2}\big)\big].
\]
Finally, as $\X = T_K(\LLL)$, we obtain
\[
\E\big[F\big(\norm{\X}_{\pvar}\big)\big] \leq C_5\E\big[F\big(|K|_{\onevar} + \norm{\LLL}_{\pvar}\big)\big] \leq C_6\E\big[F\big(\norms{[X]_\infty}^{1/2}\big)\big].
\]
\end{proof}

It was seen in~\cite[Thm.~20]{FrizShekhar17} that every (level $M$) c\`{a}dl\`{a}g $p$-rough path $\mathbf{X}$ admits a unique minimal jump extension\ $\mathbf{\bar{X}}$ as a level-$N$ rough path, $N\geq M$.
(This generalizes Lyons' fundamental extension theorem to the jump setting.)
Moreover, it is clear from the proof of~\cite[Thm.~20]{FrizShekhar17}
\begin{equation*}
\| \mathbf{\bar{X}}\|_{p\text{-var}}\lesssim \| 
\mathbf{X}\|_{p\text{-var}}
\end{equation*}
(with a multiplicative constant that depends on $N,M$). Applied with $M=2$
and Marcus lift $\mathbf{X}=\mathbf{X}\left( \omega \right)$ of an $\R^d$-valued local martingale, we obtain the following form of the BDG inequality. 

\begin{corollary}[$p$-variation level-$N$ rough path BDG]
For every $N \geq 1$, convex moderate function $F$, and $p > 2$, there exists $c, C > 0$ such that for every $\R^d$-valued local martingale $X$
\begin{equation*}
c\E\big[ F\big( \norms{\left[ X\right] _{\infty }}^{1/2}\big) \big] \leq \E\big[ F\big( \| \mathbf{\bar{X}}\|_{\pvar}\big) \big] \leq C\E\big[ F\big( \norms{\left[ X\right]_{\infty }}^{1/2}\big) \big].
\end{equation*}
\end{corollary}

This is a useful result in the study of expected signatures, which is, loosely speaking, the study of $\mathbb{E}\left[ \mathbf{\bar{X}}_{0,T}\right]$, with $\mathbf{\bar{X}}_{0,T}\in G^{N}(\R^{d}) \subset T^{N}(\R^{d})$ and component-wise expectation in the linear space $T^{N}(\R^{d})$. Since the norm of $\pi_{m}\left( \mathbf{\bar{X}}_{0,T}\right)$, the projection to $\left(\R^{d}\right) ^{\otimes m}$, is bounded (up to a constant) by $\| \mathbf{\bar{X}}\|_{p\text{-var;}\left[ 0,T\right] }^{m}$, we see that the very existence of the expected signature is guaranteed by the existence of all moments of $\left[ X\right] _{0,T}$. In a L{\'e}vy setting with triplet $\left( a,b,K\right) $, this clearly holds whenever $K\left( dy\right) 1_{\left[ \left\vert y\right\vert >1\right] }$ has moments of all orders. One can also apply this with a stopping time $T=T\left(\omega \right)$, e.g., the exit time of Brownian motion from a bounded domain. In either case, the expected signature is seen to exist (see also~\cite[Part~III]{FrizShekhar17} and~\cite{LyonsNi15} for more on this).

A motivation for the study of expected signatures comes from one of the main results of~\cite[Sec.~6]{ChevyrevLyons16} which provides a solution to the moment problem for (random) signatures, i.e., determines conditions under which the sequence of expectations $(\E[\pi_m\mathbf{\bar{X}}_{0,T}])_{m \geq 0}$ uniquely determines the law of the full signature of $\mathbf{\bar{X}}$ (see~\cite[Thm.~54]{FrizShekhar17} where the moment problem was discussed for the L{\'e}vy case, and~\cite{CassOgrodnik17, ChevyrevLyons16} for other families of random geometric rough paths).

\subsection{Convergence of semimartingales and the UCV condition}\label{subsec:convLoc}

As an application of the BDG inequality, we obtain a convergence criterion for lifted local martingales in the rough path space $(\DD^{\pvar}_\phi([0,T],G^2(\R^d)), \alpha_{\pvar})$ with a fixed path function $\phi$, which is the main result of this subsection.

We first recall the uniformly controlled variation (UCV) condition for a sequence of semimartingales $(X^n)_{n \geq 1}$. For $X \in D([0,T], \R^d)$ and $\delta > 0$, we define
\[
X^\delta_t = X_t - \sum_{s \leq t} (1-\delta/|\Delta X_s|)^+\Delta X_s.
\]
Note that $X \mapsto X^\delta$ is a continuous function on the Skorokhod space and $\sup_{t \in [0,T]}|\Delta X_t^\delta| \leq 1$ with $\Delta X^\delta_t = \Delta X_t$ whenever $|\Delta X_t| \leq \delta$.

\begin{definition}[UCV,~\cite{KP96} Definition~7.5]
We say that a sequence of semimartingales $(X^n)_{n \geq 1}$ satisfies UCV if there exists $\delta > 0$ such that for all $\alpha > 0$ there exist decompositions $X^{n,\delta} = M^{n,\delta} + K^{n,\delta}$ and stopping times $\tau^{n,\alpha}$ such that for all $t \geq 0$
\[
\sup_{n \geq 1} \PPP{\tau^{n,\alpha} \leq \alpha} \leq \frac{1}{\alpha}\; \; \textnormal{and} \; \; \sup_{n \geq 1} \E\big[M^{n,\delta}]_{t \wedge \tau^{n,\alpha}} + \norms{K^{n,\delta}}_{\onevar;[0,t\wedge \tau^{n,\alpha}]}\big] < \infty.
\]
\end{definition}

Recall the following result of Kurtz--Protter~\cite[Thm.~2.2]{KP91} (see also~\cite{KP96} Theorem~7.10 and p.~30).

\begin{theorem}\label{thm:KP91}
Let $X, (X^n)_{n \geq 1}, H, (H^n)_{n \geq 1}$ be c{\`a}dl{\`a}g adapted processes (with respect to some filtrations $\FF^n$). Suppose $(H^n, X^n)_{n \geq 1}$ converges in law (resp. in probability) to $(H,X)$ in the Skorokhod topology as $n \rightarrow \infty$, and that $(X^n)_{n \geq 1}$ is a sequence of c{\`a}dl{\`a}g semimartingales satisfying UCV. Then $X$ is a semimartingale (with respect to some filtration $\FF$) and $(H^n, X^n, \int_{0}^{\cdot} H^n_{t-} dX^n_t)$ converge in law (resp. in probability) to $(H,X,\int_0^\cdot H_{t-}dX_t)$ in the Skorokhod topology as $n \rightarrow \infty$.
\end{theorem}

We can now state the main result which allows us to pass from convergence in the Skorokhod topology to convergence in rough path topology (see also Corollary~\ref{cor:convAlpha}). 

\begin{theorem}\label{thm:UCV}
Let $(X^n)_{n \geq 1}$ be a sequence of semimartingales such that $X^n$ converges in law (resp. in probability) to a semimartingale $X$ in the Skorokhod topology. Suppose moreover that $(X^n)_{n \geq 1}$ satisfies the UCV condition. Then the lifted processes $(\X^n)_{n \geq 1}$ converge in law (resp. in probability) to the lifted process $\X$ in the Skorokhod space $D([0,T], G^2(\R^d))$. Moreover, for every $p > 2$, $(\norm{\X^n}_{\pvar})_{n \geq 1}$ is a tight collection of real random variables.
\end{theorem}

\begin{proof}
Since the stochastic area is given by the It{\^o} integral~\eqref{eq:stochArea}, the convergence in law (resp. in probability) of $(\X^n)_{n \geq 1}$ to $\X$ is an immediate consequence of Theorem~\ref{thm:KP91}.

Let $\delta > 0$ for which we can apply the UCV condition to $(X^n)_{n \geq 1}$. We next claim that $(\|\X^{n,\delta}\|_{\pvar;[0,T]})_{n \geq 1}$ is tight. Indeed, for $\varepsilon > 0$ choose $\alpha > T$ so that $1/\alpha < \varepsilon/2$. Let $X^{n,\delta} = M^{n,\delta} + K^{n,\delta}$ be the decomposition from the UCV condition along with the stopping times $\tau^{n,\alpha}$.
Then there exists $C > 0$ such that for all $n \geq 1$
\[
\PP\big[\|K^{n,\delta}\|_{\onevar;[0,T]} > C\big]
\leq \PPP{\tau^{n,\alpha} \leq \alpha} + C^{-1}\E\big[\|K^{n,\delta}\|_{\onevar;[0,T\wedge \tau^{n,\delta}]}\big] < \varepsilon,
\]
and
\begin{align*}
\sup_{n \geq 1} \PP\big[\|\M^{n,\delta}\|_{\pvar;[0,T]} > C\big]
&\leq \sup_{n \geq 1} \PPP{\tau^{n,\alpha} \leq \alpha} + C^{-2}\E\big[\|\M^{n,\delta}\|^2_{\pvar;[0,T\wedge \tau^{n,\delta}]}\big]
\\
&\leq \sup_{n \geq 1} \PPP{\tau^{n,\alpha} \leq \alpha} + C^{-2}\EEE{|[M]_{T\wedge \tau^{n,\delta}}|} 
< \varepsilon,
\end{align*}
where in the final line we used the enhanced BDG inequality Theorem~\ref{thm:pvarBDG}. Using the fact that $\X^{n,\delta} = T_{K^{n,\delta}}(\M^{n,\delta})$ proves that $(\|\X^{n,\delta}\|_{\pvar;[0,T]})_{n \geq 1}$ is tight as claimed.

To conclude, observe that $\X^n = T_{L^n}(\X^{n,\delta})$ where $L^n := X^n - X^{n,\delta}$ is a process of bounded variation for which
\[
\norms{L^n}_{\onevar;[0,T]} \leq \sum_{|\Delta X^n_t| > \delta} |\Delta X^n_t|.
\]
Since $(X^n)_{n \geq 1}$ is tight and $\sum_{|\Delta X^n_t| > \delta} |\Delta X^n_t|$ is a continuous function of $X^n$ (for the Skorokhod topology), it follows that $(\norm{\X^n}_{\pvar;[0,T]})_{n \geq 1}$ is tight as required.
\end{proof}

\begin{corollary}\label{cor:convAlpha}
Follow the notation of Theorem~\ref{thm:UCV}. Let $p > 2$ and $\phi$ an endpoint continuous, $p$-approximating path function defined on $J \subset G^2(\R^d) \times G^2(\R^d)$ such that $\X, \X^n \in \DD_{\phi}([0,T],G^2(\R^d))$ a.s.. Then for every $p' > p$, $(\X^n,\phi) \rightarrow (\X,\phi)$ in law (resp. in probability) in the metric space $(\DD^{\pprimevar}_{\phi}([0,T],G^2(\R^d)), \alpha_{\pprimevar})$.
\end{corollary}

\begin{remark}\label{remark:noArea}
As we shall see in Proposition~\ref{prop:WongZakai}, a simple way to apply Corollary~\ref{cor:convAlpha} is to assume that $\phi$ comes from the lift of a (left-invariant) path function $\phi : \R^d \to C^{\qvar}([0,T],\R^d)$ (denoted by the same symbol) which is endpoint continuous and $q$-approximating for some $1 \leq q < 2$ (so that a canonical lift indeed exists), and does not create area, i.e.,
\begin{equation}\label{eq:noArea}
\log S_2(\phi(x))_{0,1} = x \in \R^d \oplus \{0\} \subset \g^2(\R^d).
\end{equation}
Since $\X$ is Marcus-like, i.e., $\log(\Delta \X_t) \in \R^d \oplus \{0\}$, it indeed follows that $\X \in \DD_{\phi}([0,T],G^2(\R^d))$ so that $\X^\phi$ is well defined (which corresponds to interpolating the jumps of $\X$ using $\phi$); the same of courses applies to $\X^n$.
\end{remark}

\begin{proof}[Proof of Corollary~\ref{cor:convAlpha}]
Consider first the case of convergence in probability. By Theorem~\ref{thm:UCV}, $(\norm{\X^n}_{\pvar})_{n \geq 1}$ is tight, so for every $\varepsilon > 0$ we can find $R > 0$ such that
\[
\sup_{n \geq 1} \PP\big[\max\{ \norm{\X^n}, \norm{\X}, \norm{\X^n}_{\pvar}, \norm{\X}_{\pvar} \} > R\big] \leq \varepsilon.
\]
The conclusion now follows from Proposition~\ref{prop:SkorkhodAlphaMap}.
For the case of convergence in law, the proof follows in a similar way upon applying the Skorokhod representation theorem~\cite[Thm.~3.30]{Kallenberg97} to the space $D([0,T],G^N(\R^d))$ and the sequence $(\X^n)_{n \geq 1}$.
\end{proof}

As an application of Corollary~\ref{cor:convAlpha}, along with the fact that piecewise constant approximations satisfy UCV~\cite[Ex.~3.7]{KP91}, we obtain the following Wong--Zakai-type result
(which shall be substantially generalized in Section~\ref{subsec:WongZakai} using different methods).

\begin{remark}
The following result resembles the Wong--Zakai theorem of~\cite{KPP95}, where it was shown that ODEs driven by approximations of the form $X^h_t = h^{-1} \int_{t-h}^t X_s ds$ converge to an SDE of the Marcus type.
Here we are able to complement~\cite{KPP95} by showing this for the case of genuine piecewise linear (and a variety of other) approximations.
Moreover the deterministic nature of our rough path approach is able to handle anticipating initial data (see Remark~\ref{rem:anticipating}).
\end{remark}

\begin{proposition}[Wong--Zakai with no area]\label{prop:WongZakai}
Let $X$ be a semimartingale with Marcus lift $\X$, and let $\D_n \subset [0,T]$ be a sequence of deterministic partitions such that $|\D_n| \rightarrow 0$. Let $X^{[D_n]}$ be the piecewise constant approximations of $X$ along the partition $\D_n$, and let $\X^{[\D_n]}$ be their (Marcus) lifts.

Let $\phi : \R^d \to C^{\qvar}([0,1], \R^d)$ be an endpoint continuous, $q$-approximating path function for some $1 \leq q < 2$ such that $\phi$ does not create area, i.e.,~\eqref{eq:noArea} holds.
By an abuse of notation, let $\phi : \exp(\R^d \oplus \{0\}) \to C^{\qvar}([0,1],G^2(\R^d))$ denote also the canonical lift of $\phi$, treated as a path function defined on $\exp(\R^d \oplus \{0\}) \subset G^2(\R^d)$.

\begin{enumerate}[label=\upshape(\arabic*\upshape)]
\item \label{point:WZ1} Consider the admissible pairs $(\X^{[\D_n]},\phi)$ and $(\X,\phi)$ in $\DD([0,T],G^2(\R^d))$. Then for every $p > 2$, $\alpha_{\pvar}(\X^{[\D_n]}, \X) \rightarrow 0$ in probability as $n \rightarrow \infty$.

\item \label{point:WZ2} Let $X^{\D_n,\phi}$ be the piecewise-$\phi$ interpolation of $X$ along the partition $\D_n$. Let $U^\x_{T \leftarrow 0} \in \Diff^m(\R^e)$ denote the flow associated to the RDE~\eqref{eq:cadlagRDE}. Then $U^{X^{\D_n,\phi}}_{T \leftarrow 0}$ converges in probability to $U^{\X^\phi}_{T \leftarrow 0}$ as $n \rightarrow \infty$ (as $\Diff^m(\R^e)$-valued random variables).
\end{enumerate}
\end{proposition}

\begin{proof}
\ref{point:WZ1} follows immediately from Corollary~\ref{cor:convAlpha} and the fact that $(X^{[\D_n]})_{n \geq 1}$ satisfies UCV~\cite[Ex.~3.7]{KP91}.

For~\ref{point:WZ2}, note that, since $\phi$ does not create area, $(\X^{[\D_n]})^{\phi}$ coincides (up to reparametrization) with the canonical lift of $X^{\D_n,\phi}$. The conclusion now follows from~\ref{point:WZ1} and Theorem~\ref{thm:alphaMetric}.
\end{proof}

We now record a relation between canonical RDEs and Marcus SDEs.

\begin{proposition}\label{prop:Marcus}
Let $X : [0,T] \to \R^d$ be a semimartingale and $\X$ its Marcus lift. Then for vector fields $V = (V_1,\ldots, V_d)$ in $\Lip^\gamma(\R^e)$ for some $\gamma > 2$, it holds that the canonical RDE
\begin{equation}\label{eq:cadlagRDE2}
dY_t = V(Y_t) \diamond d\X_t, \; \; Y_0 \in \R^e
\end{equation}
(i.e., the path function $\phi$ is the taken to be log-linear) coincides a.s. with the Marcus SDE
\[
dY_t = V(Y_t)\diamond dX_t, \; \; Y_0 \in \R^e.
\]
\end{proposition}

\begin{proof}
Recall that, by definition, the Marcus SDE satisfies~\cite{KPP95}
\begin{align*}
Y_t = Y_0 + \int_0^t V(Y_{s-})dX_s + \frac{1}{2}\int_0^t V' V(Y_s)d[X]^c_s
+ \sum_{0 < s \leq t}\left\{ e^{V\Delta X_s} (Y_{s-}) - Y_{s-} - V(Y_{s-})\Delta X_s \right\},
\end{align*}
where $e^{W}(y)$ denotes the flow at time $1$ along the vector field $W$ from $y$, i.e., the solution at $t=1$ to $z_0 = y$, $\dot z_t = W(z_t)$.

Recall likewise that, by definition, the rough canonical equation (in the sense of~\cite[Def.~37]{FrizShekhar17}) satisfies
\begin{align*}
Y_t = Y_0 + \int_0^t V(Y_{s-})d\X_s + \sum_{0 < s \leq t}\left\{ e^{V\Delta X_s} (Y_{s-}) - Y_{s-} - V(Y_{s-})\Delta X_s - V' V(Y_{s-})\frac{1}{2}(\Delta X_s)^{\otimes 2}\right\},
\end{align*}
which agrees with the solution to the canonical RDE~\eqref{eq:cadlagRDE2} from Definition~\ref{def:canonicalRDE} (see~\cite[Thm.~38]{FrizShekhar17}; we point out that~\cite{KPP95} and~\cite{FrizShekhar17} are not restricted to the case of finite activity).
It remains to verify that a.s.
\[
\int_0^t V(Y_{s-})d\X_s =\int_{0}^t V(Y_{s-})dX_s + \frac{1}{2}V'V(Y_s)d[X]^c_s  + \sum_{0<s\leq t} V'V(Y_{s-})\frac{1}{2}(\Delta X_s)^{\otimes 2}.
\]
To this end, observe that
\[
\int_{0}^t V(Y_{s-}) d\X_s = \lim_{|\D| \rightarrow 0} \sum_{t_i \in \D} V(Y_{t_i-})X_{t_i,t_{i+1}} + V'V(Y_{t_i-})\X^{(2)}_{t_i,t_{i+1}}\;,
\]
where $\X^{i,j}_{s,t} = \int_s^t X^i_{s,u-}dX^j_u +\frac{1}{2}[X^i,X^j]^c_{s,t}+\frac{1}{2}\sum_{r \in (s,t]}\Delta X^i_r\Delta X^j_r$.
It follows from Lemma~\ref{lem:areaConv} that for a (deterministic) sequence of partitions with $|\D_n| \to 0$, we have a.s.
\[
\int_0^t V(Y_{s-})d\X_s = \lim_{n\to \infty} \sum_{t_i \in \D_n}V(Y_{t_i-})dX_{t_i,t_{i+1}} +\frac{1}{2}V'V(Y_{t_i-})[X]_{t_{i},t_{i+1}} + \sum_{r\in (t_i,t_{i+1}]}V'V(Y_{t_i-})\frac{1}{2}(\Delta X_r)^{\otimes 2}
\]
from which the conclusion readily follows.
\end{proof}

\begin{remark}[Marcus SDEs with piecewise constant driver] \label{remark:MDEconst}
Observe that a simple special case of Proposition~\ref{prop:Marcus} (which does not require any probabilistic considerations) is a piecewise constant path $X^{[\D]}$, which is constant between the points of a partition $\D \subset [0,T]$. In this case, the
solution to 
\[
dY =V(Y) \diamond dX^{[\D]}, \; \; Y_0 \in \R^e,
\]
agrees, for all $t \in \D$, with the ODE solution 
\[
d\tilde Y =V(\tilde Y) dX^{\D}, \; \; Y_0 \in \R^e,
\]
where $X^{\D}$ is the piecewise linear path obtained from $X^{[\D]}$ by connecting with a straight line consecutive points $X^{[\D]}_{t_n} \rightsquigarrow X^{[\D]}_{t_{n+1}}$ for all $t_n \in \D$.  
\end{remark}

We are now ready to state the precise criterion for convergence in law (resp. in probability) of Marcus SDEs which was advertised in the introduction and which is analogous to the same criterion for It{\^o} SDEs~\cite[Thm.~5.4]{KP91}.

\begin{theorem}\label{thm:MarcusSDEs}
Let $V = (V_1,\ldots, V_d)$ be a collection of $\Lip^\gamma$ vector fields on $\R^e$ for some $\gamma > 2$. Let $Y_0, (Y_0^n)_{n \geq 1}$ be a collection of (random) initial conditions in $\R^e$ and $X, (X^n)_{n \geq 1}$ be a collection of semimartingales such that $(X^n)_{n \geq 1}$ satisfies UCV and $(Y^n_0, X^n) \rightarrow (Y_0, X)$ in law (resp. in probability) as $n \rightarrow \infty$ (as $\R^e \times D([0,T], \R^d)$-valued random variables). Then the solutions to the Marcus SDEs
\[
dY^n_t = V(Y^n_t)\diamond dX^n_t, \; \; Y^n_0 \in \R^e,
\]
converge in law (resp. in probability) as $n \rightarrow \infty$ (in the Skorokhod topology) to the solution of the Marcus SDE
\begin{equation} \label{eq:MSDE} 
dY_t = V(Y_t)\diamond dX_t, \; \; Y_0 \in \R^e.
\end{equation}
\end{theorem}

\begin{proof}
This is an immediate consequence of Corollary~\ref{cor:convAlpha}, Proposition~\ref{prop:Marcus}, and the deterministic continuity of the solution map (part~\ref{point:contSkor2} of Proposition~\ref{prop:alphaSkorokhod}).
\end{proof}

\begin{remark}\label{rem:anticipating}
We have not been explicit about filtrations, but of course, every semimartingale $X^n$ above is adapted to some filtration $\{ \FF^n_t \}_{t \geq 0}$. In the same vain, as is standard in the context of SDEs, the initial datum $Y_0^n$ is assumed to be $\FF^n_0$-measurable, so that~\eqref{eq:MSDE} makes sense as a bona fide integral equation (as recalled in the proof of Proposition \ref{prop:Marcus}).

Situations where $Y^0_n$ is independent of the driving noise $X^n$ are then immediately handled. If, on the other hand, $Y^0_n$ depends in some anticipating fashion on the driving noise, then classical SDE theory (Marcus or It{\^o}) breaks down and ideas from anticipating stochastic calculus are necessary (such as composing the stochastic flow with anticipating initial data; in the Marcus context this would be possible thanks to~\cite[Thm.~3.4]{KPP95}). Our 
(essentially deterministic) rough path approach bypasses such problems entirely. We shall not pursue further application of rough paths to ``anticipating Marcus SDEs'' here, but note that this could be done analogously to~\cite{CFV07}.
\end{remark}

\subsection{Examples}

We now give a list of examples to which Corollary~\ref{cor:convAlpha}, Proposition~\ref{prop:WongZakai}, and Theorem~\ref{thm:MarcusSDEs} apply. The main criterion of application is of course the UCV condition. For further examples of sequences of semimartingales satisfying UCV, see~\cite[Sec.~3]{KP91}.
We note that in the framework of Theorem~\ref{thm:MarcusSDEs}, the UCV condition cannot in general be ommited (but see Theorem~\ref{thm:WongZakai} below) as seen, e.g., in homogenization theory~\cite{LL03, FGL15, KM16} and non-standard approximations to Brownian motion~\cite{McShane72, FrizOberhauser09}.

\begin{example}[Piecewise constant approximations]\label{ex:pieceConst}
Let $X$ be a c{\`a}dl{\`a}g semimartingale and $X^{[\D_k]}$ be its piecewise constant approximation (see Figure~\ref{fig:piecewiseConst}) along a sequence of deterministic partitions $\D_k \subset [0,T]$ such that $|\D_k| \rightarrow 0$. Then by Theorem~\ref{thm:MarcusSDEs}, the solutions to
\[
dY^{[\D_k]}_t = V(Y^{[\D_k]}_t) \diamond dX^{[\D_k]}_t, \; \; Y^{[\D_k]}_0 = y_0 \in \R^e
\]
converge in probability (for the Skorokhod topology) to the solution of
\[
dY_t = V(Y_t) \diamond dX_t, \; \; Y_0 = y_0 \in \R^e.
\]
Moreover, if $X$ is continuous, then $Y^{[\D_k]}$ converges in probability for the uniform topology on $[0,T]$ to $Y$ (which is now also the solution to the Stratonovich SDE).
\end{example}

\begin{figure}
\centering
\begin{minipage}{0.45\textwidth}
\caption{C\`adl\`ag path (blue) with piecewise constant approximtation (red)} \label{fig:piecewiseConst}
\smallskip
\begin{tikzpicture}[scale=0.75]
\begin{axis}
[ xtick={0,10}, ytick={0},  xticklabels={,,},  yticklabels={,,},  xlabel={$t$}, ylabel={$x$}, xlabel style={below right}, ylabel style={above left},
  xmin= -1, xmax=11,  ymin=-3,  ymax=10]

\addplot[domain=0:3,blue, thick] {7-(x-3)^2};   
\addplot[domain=3:6,blue, thick, samples = 200] {5 + sin(  (x-3)/(6-3) * 3 *360 ) }; 
\addplot[domain=6:10,blue, thick, samples= 200] {8-(x-6)/2 - sin(  (x-6)/(10-6) * 1 *360 ) / 3 };

\addplot[domain=0:2, jump mark left, thick, red, samples=2] {7-(x-3)^2};
\addplot[domain=2:4, jump mark left, thick, red, samples=2] {7-(2-3)^2 + (x-2)*(-1/3)};
\addplot[domain=4:6, jump mark left, red, thick, samples=2] {5};
\addplot[domain=6:10, jump mark left, red, thick, samples=3] {8-(x-6)/2};

\addplot[soldot] coordinates{(3,5)(6,8)};   
\addplot[holdot] coordinates{(3,7)(6,5)};  

\addplot[soldotred] coordinates{(0,-2)(2,6)(4,5)(6,8)(8,7)(10,6)};
\addplot[holdotred] coordinates{(2,-2)(4,6)(6,5)(8,8)(10,7)};   

\end{axis}
\end{tikzpicture}
\end{minipage}
\hfill
\begin{minipage}{0.45\textwidth}
\centering
\captionsetup{width=0.9\textwidth}
\centering
\caption{C\`adl\`ag path (blue) with piecewise linear approximation (red, dashed)} \label{fig:piecewiseLinear}
\smallskip
\begin{tikzpicture}[scale=0.75]
\begin{axis}
[ xtick={0,10}, ytick={0},  xticklabels={,,},  yticklabels={,,},  xlabel={$t$}, ylabel={$x$}, xlabel style={below right}, ylabel style={above left},
  xmin= -1, xmax=11,  ymin=-3,  ymax=10]

\addplot[domain=0:3,blue, thick] {7-(x-3)^2};   
\addplot[domain=3:6,blue, thick, samples = 200] {5 + sin(  (x-3)/(6-3) * 3 *360 ) }; 
\addplot[domain=6:10,blue, thick, samples= 200] {8-(x-6)/2 - sin(  (x-6)/(10-6) * 1 *360 ) / 3 };

\addplot[domain=0:2, red, thick, densely dashed, samples=2] {7-(x-3)^2};
\addplot[domain=2:4, red, thick, densely dashed, samples=2] {7-(2-3)^2 + (x-2)*(-1/2)};
\addplot[domain=4:6, red, thick, densely dashed, samples=2] {5+3*(x-4)/2};
\addplot[domain=6:8, red, thick, densely dashed, samples=2] {8-(x-6)/2)};
\addplot[domain=8:10, red, thick, densely dashed, samples=2] {8-(x-6)/2};

\addplot[soldot] coordinates{(0,-2)(3,5)(6,8)};
\addplot[holdot] coordinates{(3,7)(6,5)(10,6)};

\addplot[soldotred] coordinates{(0,-2)(2,6)(4,5)(6,8)(8,7)(10,6)};

\end{axis}
\end{tikzpicture}
\end{minipage}
\end{figure}

\begin{example}[Piecewise linear approximations]\label{ex:pieceLinear}
Let $X$ be a c{\`a}dl{\`a}g semimartingale and now let $X^{\D_k}$ be its piecewise linear (i.e., Wong--Zakai) approximation (see Figure~\ref{fig:piecewiseLinear}) along a sequence of deterministic partitions $\D_k \subset [0,T]$ such that $|\D_k| \rightarrow 0$. Consider the solutions to random ODEs
\[
dY^{\D_k}_t = V(Y^{\D_k}_t) dX^{\D_k}_t, \; \; Y^{\D_k}_0 = y_0 \in \R^e,
\]
and the Marcus SDE
\[
dY_t = V(Y_t) \diamond dX_t, \; \; Y_0 = y_0 \in \R^e.
\]
Then, by Proposition~\ref{prop:WongZakai}, it holds that $Y^{\D_k}_T \rightarrow Y_T$ in probability.

Moreover, if $X$ is continuous, then in light of the last part of Example~\ref{ex:pieceConst} and Remark~\ref{remark:MDEconst}, $Y^{\D_k}$ converges in probability for the uniform topology on $[0,T]$ to $Y$ (which, we emphasize again, is now the solution to the Stratonovich SDE), which agrees with the classical Wong--Zakai theorem for continuous semimartingales.
\end{example}

\begin{example}[Donsker approximations to Brownian motion]\label{ex:DonskerBrown}
Consider an $\R^d$-valued random walk $X^{n}$ with iid increments and finite second moments, rescaled so that $X^{n} \rightarrow B$ in law. Here we treat $X^n$ as either piecewise constant or interpolated using any sufficiently nice path function $\phi$ which does not create area, i.e., satisfies~\eqref{eq:noArea} (e.g., piecewise linear). Then $(X^n)_{n \geq 1}$ satisfies UCV, so by Corollary~\ref{cor:convAlpha} we again have convergence of the Marcus SDEs (or random ODEs in case of continuous interpolations)
\[
dY^n_t = V(Y^n_t) \diamond dX^n_t \; \; Y_0 = y_0 \in \R^e
\]
in law for the uniform topology on $[0,T]$ to the Stratonovich limit
\[
dY_t = V(Y_t)\circ dB_t \; \; Y_0 = y_0 \in \R^e.
\]
This is a special case of~\cite[Ex.~5.12]{Chevyrev15} (see also Example~\ref{ex:DonskerLevy}) which improves the main result of Breuillard et al.~\cite{Breuillard09} in the sense that no additional moment assumptions are required (highlighting a benefit of $p$-variation vs. H{\"o}lder topology).
\end{example}

\begin{example}[Null array approximations to L{\'e}vy processes]\label{ex:DonskerLevy}
Generalizing Example~\ref{ex:DonskerBrown}, consider a null array of $\R^d$-valued random variables $X_{n1},\ldots, X_{nn}$, i.e., $\lim_{n \rightarrow \infty} \sup_k \EEE{|X_{nk}|\wedge 1} = 0$, and, for every $n \geq 1$, $X_{n1},\ldots, X_{nn}$ are independent. Consider the associated random walk
\[
X^n : [0,1] \to \R^d, \; \; X^n_t = \sum_{k=1}^{\floor{tn}} X_{nk}.
\]
Suppose $X^n \rightarrow X$ in law for a L{\'e}vy process $X$ (see~\cite[Thm.~13.28]{Kallenberg97} for necessary and sufficient conditions for this to occur), which in particular implies that for some $h > 0$
\begin{enumerate}
\item $\sum_{k=1}^n \E[X_{nk} \1{|X_{nk}| < h}] \rightarrow b^h$,

\item $\sum_{k=1}^n \E[X^i_{nk}X^j_{nk} \1{|X_{nj}| < h}] \rightarrow a^h_{i,j}$, and

\item $\sum_{k=1}^n \E[f(X_{nk})] \rightarrow \nu(f)$ for every $f \in C_b(\R^d)$ which is identically zero on a neighbourhood of zero,
\end{enumerate}
where $b^h$, $a^h$, and $\nu$ are determined by the L{\'e}vy triplet of $X$ (in particular $\nu$ is the L{\'e}vy measure of $X$).
As a consequence, it is immediate to verify that $(X^n)_{n \geq 1}$ satisfies UCV.
By Theorem~\ref{thm:MarcusSDEs}, the solutions to
\[
dY^{n}_t = V(Y^{n}_t) \diamond dX^{n}_t, \; \; Y^{n}_0 = y_0 \in \R^e,
\]
converge in law (for the Skorokhod topology) to the solution of the Marcus SDE
\[
dY_t = V(Y_t) \diamond dX_t, \; \; Y_0 = y_0 \in \R^e.
\]
If, once more, $X^n$ are interpolated using any sufficiently nice path function $\phi$ (which in particular does not create area~\eqref{eq:noArea}), an application of Corollary~\ref{cor:convAlpha} implies that $Y^n_T$ (now solutions to random ODEs) converge in law to $Y_T$ (now the solution to the random RDE driven by $\X^\phi$). Note that if $\phi$ is allowed to create area and $X_{n1},\ldots, X_{nn}$ are further assumed iid for every $n \geq 1$, then this is precisely the case addressed in~\cite[Ex.~5.12]{Chevyrev15} (though in this case one must consider a non-Marcus lift of $X$ similar to the upcoming Theorem~\ref{thm:WongZakai}).
\end{example}

\begin{example}[Martingale CLT] Let $(X^{n})_{n \geq 1}$ be a sequence of $\R^d$-valued c{\`a}dl{\`a}g local martingales. Suppose that, as $n\rightarrow \infty $, 
\[
\E\big[\sup_{t\in (0,T]} |\Delta X_{t}^{n}| \big] \rightarrow 0\text{ 
},\,\,\,\,\,\,\left[ X^{n},X^{n}\right] _{t}\rightarrow C\left( t\right) 
\text{ }\forall t\in (0,T],
\]
where $t\mapsto C\left( t\right) \in \mathbb{R}^{d\times d}$ is continuous and deterministic. Then $X^n \rightarrow X$, where $X$ is a continuous Gaussian process with independent increments and $\EEE{X(t)X(t)^T} = C(t)$~\cite[Thm.~1.4, p.~339]{EK86}, and moreover the UCV condition is satisfied~\cite[p.~26]{KP96}. Therefore solutions to
\[
dY^n_t = V(Y^n_t) \diamond dX^n_t, \; \; Y^n_0 = y_0 \in \R^e,
\]
converge in law for the uniform topology on $[0,T]$ to the solution of the Stratonovich SDE
\[
dY_t = V(Y_t) \circ dX_t, \; \; Y_0 = y_0 \in \R^e.
\]
\end{example}

\subsection{Wong--Zakai revisited} \label{subsec:WongZakai}

In this subsection we significantly expand Proposition~\ref{prop:WongZakai} by showing convergence in probability of very general (area-creating) interpolations of c{\`a}dl{\`a}g semimartingales. 
If the interpolation creates area, we in general no longer expect to converge to the Marcus lift of $X$ (which is the reason one cannot apply Proposition~\ref{prop:WongZakai}), and therefore we first modify the lift $\X$ appropriately.

Throughout the section, we fix a (left-invariant) $q$-approximating path function $\phi : \R^d \to C^{\qvar}([0,1], \R^d)$ for some $1 \leq q < 2$ (which we take to be defined on the entire space $\R^d$ only for simplicity). Consider the two maps
\begin{align*}
&\psi : \R^d \to G^2(\R^d)
\\
&a : \R^d \to \g^{(2)}(\R^d) \cong \mathfrak{so}(\R^d)
\end{align*}
defined uniquely by
\[
S_2(\phi(x))_{0,1} = \psi(x) = \exp(x + a(x))
\]
(so that $a(x)$ is the area generated by the path $\phi(x) : [0,1] \to \R^d$).

Let us also fix a c{\`a}dl{\`a}g semimartingale $X : [0,T] \to \R^d$ and a sequence of deterministic partitions $\D_k \subset [0,T]$ such that $\lim_{k \rightarrow \infty} |\D_k| = 0$.
Consider the following assumption.

\begin{assumption}\label{assump:VAssumption}
There exists a c{\`a}dl{\`a}g bounded variation process $B : [0,T] \to \g^{(2)}(\R^d)$ such that
\begin{equation*}
\sup_{t \in [0,T]} \Big|B_t - \sum_{\D_k \ni t_j \leq t} a(X_{t_{j}} - X_{t_{j-1}})\Big| \rightarrow 0 \; \; \textnormal{in probability as $k \rightarrow \infty$}.
\end{equation*}
\end{assumption}

Before stating the Wong--Zakai theorem, we give several examples of $\phi$ for which Assumption~\ref{assump:VAssumption} is satisfied.

\begin{example}[No area]
If $\phi(x)$ does not create area for all $x \in \R^d$, so that $a \equiv 0$ (e.g., when $\phi$ is the linear interpolation on $\R^d$), then evidently Assumption~\ref{assump:VAssumption} is satisfied with $B_t \equiv 0$.
\end{example}

\begin{example}[Hoff-type process] \label{ex:Hoff}
Suppose that $\R^d = \R^2$, so that $\g^{(2)}(\R^2) \cong \R$. Let $\phi$ travel to $(x,y)$ first linearly along the $x$-coordinate and then linearly along the $y$-coordinate:
\[
\phi(x,y)_t = \1{t \in [0,1/2]} 2tx + \1{t \in (1/2,1]} (x + (2t-1)y).
\]
Then $a(x,y) = \frac{1}{2}xy$, so that Assumption~\ref{assump:VAssumption} is satisfied with $B_t = \frac{1}{2}[X^1,X^2]_t$.
\end{example}

\begin{example}[Regular $a$]
Suppose more generally that $a$ is twice differentiable at $0$, so that by Lemma~\ref{lem:quadDecay}
\[
a(X_{s,t}) = \frac{1}{2}D^2a(0) (X_{s,t}^{\otimes 2}) + o(|X_{s,t}|^2).
\]
We see in this case that Assumption~\ref{assump:VAssumption} is satisfied with
\[
B_t = \frac{1}{2} \sum_{i,j = 1}^d D^2a(0)^{i,j} [X^i,X^j]^c_t + \sum_{s \leq t} a(\Delta X_s).
\]
\end{example}

For a partition $\D \subset [0,T]$, let $X^{\D,\phi}$ be the piecewise-$\phi$ interpolation of $X$ along $\D$, and $\X^{\D,\phi}$ its canonical lift.  The following is the main result of this subsection.

\begin{theorem}[Wong--Zakai]\label{thm:WongZakai}
Suppose that Assumption~\ref{assump:VAssumption} is satisfied. Let $\bar \X : [0,T] \to G^2(\R^d)$ be the modified level-2 lift of $X$ defined by
\[
\bar \X := \exp(X + \bar A), \; \;  \bar A_t := A_t + B_t,
\]
where $\X = \exp(X + A)$ is the Marcus lift of $X$.
Suppose further that $\phi$ is endpoint continuous.

\begin{enumerate}[label=\upshape(\arabic*\upshape)]
\item \label{point:thm1} Consider the admissible pair $(\bar \X,\phi) \in \DD([0,T],G^2(\R^d))$. Then for every $p > 2$, it holds that
\begin{equation*}
\alpha_{\pvar}(\X^{\D_k,\phi},\bar \X) \rightarrow 0 \; \; \textnormal{in probability as $k \rightarrow \infty$}.
\end{equation*}

\item \label{point:thm2} Let $U^\x_{T \leftarrow 0} \in \Diff^m(\R^e)$ denote the flow associated to the RDE~\eqref{eq:cadlagRDE}.
Then $U^{X^{\D_k,\phi}}_{T \leftarrow 0}$ converges in probability to $U^{\bar \X^\phi}_{T \leftarrow 0}$ (as a $\Diff^m(\R^e)$-valued random variable).
\end{enumerate}
\end{theorem}

\begin{remark}\label{remark:phi}
Note that the jumps of $B_t$ must necessarily be of the form $a(\Delta X_t)$. Hence $\Delta \bar\X_t \in \psi(\R^d)$, so that indeed $\bar \X \in \DD_\phi([0,T],G^2(\R^d))$ and $\bar\X^{\phi}$ is well-defined (this is an abuse of notation since $\phi$ is a path function on $\R^d$, but because $\phi(x)$ of finite $q$-variation, it can be canonically lifted to a path function $\phi : \psi(\R^d) \to G^2(\R^d)$).
\end{remark}

For the proof of the theorem, we require several lemmas.

\begin{lemma}\label{lem:quadDecay}
Let $\eta_{\qvar}$ be a $q$-variation modulus of $\phi$ (see Definition~\ref{def:pvar}). Then
\[
|a(x)| \leq \eta_{\qvar}(r)|x|^2, \; \; \forall r > 0, \; \; \forall x \in \R^d \textnormal{ s.t. } |x| \leq r. 
\]
\end{lemma}

\begin{proof}
This is immediate from the property
\[
|x| + |a (x)|^{1/2} \asymp \norm{\psi(x)} \leq \norm{\phi(x)}_{\qvar;[0,1]} \leq \eta_{\qvar}(r)|x|.
\]
\end{proof}

The following lemma essentially involves no probability.

\begin{lemma}\label{lem:smallOsc}
Suppose $\phi$ is endpoint continuous. Then for a.e. sample path $\bar\X \in D([0,T], G^2(\R^d))$, it holds that for all $\varepsilon > 0$ there exists $r > 0$,
such that for all partitions $\D \subset [0,T]$ with $|\D| < r$, there exists $\delta_0 > 0$,
such that for all $\delta < \delta_0$, there exists $\lambda \in \Lambda$
such that
\begin{align*}
&|\lambda| < 2|\D|,
\\
&\bar\X_{t_n} = \bar\X^{\phi,\delta}_{\lambda(t_n)}, \; \; \forall t_n \in \D, \; \; \textnormal{ and}
\\
&\max_{t_n \in \D} \sup_{t \in [t_n,t_{n+1}]} d(\X^{\D,\phi}_{t_n,t}, \bar\X^{\phi,\delta}_{\lambda(t_n),\lambda(t)}) < \varepsilon.
\end{align*}
\end{lemma}

\begin{proof}
Since $\bar\X$ is c{\`a}dl{\`a}g, for every $\varepsilon > 0$ we can find $r > 0$ sufficiently small and a partition $\calP = (s_0,\ldots,s_m)$ so that $|s_{i+1} - s_i| > 2r$ and
\[
\sup_{u,v \in [s_{i},s_{i+1})} \|\bar\X_{u,v}\| < \varepsilon.
\]
Then whenever $|\D| < r$, for every $t_n \in \D$, there exists at most one $s_i \in \calP$ such that $s_i \in [t_n,t_{n+1}]$.
Now using the fact that $\phi$ is $q$-approximating and endpoint continuous, the claim readily follows.
\end{proof}

\begin{lemma}\label{lem:pvarTight}
The family of real random variables $(\|\X^{\D,\phi}\|_{\pvar;[0,T]})_{\D \subset [0,T]}$ is tight.
\end{lemma}

\begin{remark}
As the proof of Lemma~\ref{lem:pvarTight} will reveal, the biggest difficulty is overcome by the enhanced BDG inequality for c{\`a}dl{\`a}g local martingales (Theorem~\ref{thm:pvarBDG}). We wish to emphasize that the lemma is even helpful in the context of a rough paths proof of the Wong--Zakai theorem for \emph{continuous} semimartingales with piecewise linear interpolations (so that $a \equiv 0$), since an analogous tightness result is still needed in this case and is non-trivial (cf.~\cite[Thm.~14.15]{FrizVictoir10}).

We also mention that part~\ref{point:WZ1} of Proposition~\ref{prop:WongZakai} in particular shows that $(\|\X^{[\D_k]}\|_{\pvar})_{k \geq 1}$ is tight, which can significantly simply the proof of Lemma~\ref{lem:pvarTight} (at least if one restricts attention to the family $(\|\X^{\D_k,\phi}\|_{\pvar})_{k \geq 1}$). However we give a direct proof of the general result here.
\end{remark}

\begin{proof}[Proof of Lemma~\ref{lem:pvarTight}]
We can decompose $X = L + K$ (non-uniquely) where $K$ is of bounded variation and $L$ is a local martingale with jumps bounded by some $M > 0$ (e.g.,~\cite[Prop.~4.17, p.~42]{JacodShiryaev03}). Using the localizing sequence $\tau_m = \inf \{ t \in [0,T] \mid |L_t| + |K_t| > m \}$, we may further suppose that $L$ is bounded.

For a partition $\D \subset [0,T]$, consider the piecewise constant path $\tilde \X^{[\D]} : [0,T] \to G^2(\R^d)$ which is constant on $[t_n,t_{n+1})$ and $\tilde \X^{[\D]}_{t_n} = \X^{\D,\phi}_{t_n}$ for every $t_n \in \D$. Consider also the piecewise constant semimartingale $X^{[\D]}$ which is constant on $[t_n,t_{n+1})$ and $X^{[\D]}_{t_n} = X_{t_n}$ for every $t_n \in \D$. Let $\X^{[\D]}$ be the (Marcus) lift of $X^{[\D]}$. By definition of $a : \R^d \to \g^{(2)}(\R^d)$, we have for all $t_n \in \D$
\[
\tilde \X^{[\D]}_{t_n} = \X^{[\D]}_{t_n} \otimes \exp\Big(\sum_{\D \ni t_k \leq t_n} a(X_{t_k} - X_{t_{k-1}})\Big).
\]
It follows that
\[
\|\tilde \X^{[\D]}\|^p_{\pvar;[0,T]} \leq C\big(\|\X^{[\D]}\|^p_{\pvar;[0,T]} + \|Y^{[\D]}\|_{\var{p/2};[0,T]}^{p/2}\big),
\]
where $Y^{[\D]}_t = \sum_{\D \ni t_n \leq t} a(X_{t_n} - X_{t_{n-1}}) \in \g^{(2)}(\R^d)$.

Observe that $\X^{\D,\phi}$ is a reparametrization of $(\tilde \X^{[\D]})^{\phi}$ (where we use the same abuse of notation as in Remark~\ref{remark:phi}), so in particular
\[
\|\X^{\D,\phi}\|_{\pvar;[0,T]} = \|(\tilde \X^{[\D]})^{\phi}\|_{\pvar;[0,T]}. 
\]
Moreover, defining the piecewise constant local martingale $L^{[\D]}$, its lift $\LLL^{[\D]}$, and the piecewise constant path of bounded variation $K^{[\D]}$ in the same way as $X^{[\D]}$, we note that $\X^{[\D]} = T_{K^{[\D]}}(\LLL^{[\D]})$.
Following Lemma~\ref{lem:pvarJumps}, it suffices to show that the families
\[
(\|Y^{[\D]}\|_{\var{p/2}})_{\D \subset [0,T]}, \; \; (\|\LLL^{[\D]}\|_{\pvar})_{\D \subset [0,T]}, \; \; (\|K^{[\D]}\|_{\onevar})_{\D \subset [0,T]}
\]
are tight. This in turn follows respectively from Lemma~\ref{lem:quadDecay}, the enhanced BDG inequality Theorem~\ref{thm:pvarBDG}, and the fact that $\|K^{[\D]}\|_{\onevar} \leq \norm{K}_{\onevar}$.
\end{proof}

\begin{proof}[Proof of Theorem~\ref{thm:WongZakai}]
Note that~\ref{point:thm2} follows directly from~\ref{point:thm1} and Theorem~\ref{thm:alphaMetric}. To show~\ref{point:thm1}, note that $\|\bar \X\|_{\pvar} < \infty$ a.s. (by Corollary~\ref{cor:semimartPvar}), so following Lemma~\ref{lem:pvarTight} and interpolation (Lemma~\ref{lem:interBeta}), it suffices to show that to show that
\begin{equation*}
\alpha_{\infty}(\X^{\D_k,\phi},\bar \X) \rightarrow 0 \; \; \textnormal{in probability as $k \rightarrow \infty$}.
\end{equation*}

Let $\bar A, \bar A^{\phi,\delta}$ and $A^{\D_k,\phi}$ denote the stochastic area of $\bar \X, \bar\X^{\phi,\delta}$ and $\X^{\D_k,\phi}$ respectively. We have for $t \in [t_n,t_{n+1}] \subset \D_k$
\[
\X^{\D_k,\phi}_{t} = \psi(X_{0,t_1}) \ldots \psi(X_{t_{n-1},t_n}) S_2(\phi(X_{t_n,t_{n+1}}))_{\frac{t-t_n}{t_{n+1}-t_n}},
\]
and so by the Campbell--Baker--Hausdorff formula
\[
A^{\D_k,\phi}_{t} = A^{\D_k,\phi}_{t_n,t} + \sum_{j=0}^{n-1} a(X_{t_j,t_{j+1}}) + \frac{1}{2}\big[X_{0,t_n}, X^{\D_k,\phi}_{t_n,t}\big] + \frac{1}{2}\sum_{j=0}^{n-1} [X_{0,t_j},X_{t_j,t_{j+1}}].
\]
Likewise
\[
\bar A_{t} = \bar A_{t_n,t} + \sum_{j=0}^{n-1} \bar A_{t_j,t_{j+1}} + \frac{1}{2}\left[X_{0,t_n}, X_{t_n,t}\right] + \frac{1}{2}\sum_{j=0}^{n-1} [X_{0,t_j},X_{t_j,t_{j+1}}].
\]
Recalling further that $\bar A_{t_j,t_{j+1}} = A_{t_j,t_{j+1}} + B_{t_j,t_{j+1}}$ and $X^{\D_k,\phi}_{t_n} = X_{t_n}$, it follows that for all $t_n \in \D_k$
\begin{equation*}\label{eq:partPointsBound}
d(\X^{\D_k,\phi}_{t_n}, \bar\X_{t_n}) \asymp \Big|\sum_{j=0}^{n-1} A_{t_j,t_{j+1}} + B_{t_j,t_{j+1}} - a(X_{t_{j+1}} - X_{t_j})\Big|^{1/2}.
\end{equation*}
Combining Lemma~\ref{lem:smallOsc} with Assumption~\ref{assump:VAssumption}, we see that the proof is complete once we show
\[
\max_{t_n \in \D_k} \Big|\sum_{j=0}^{n-1} A_{t_j,t_{j+1}}\Big| \rightarrow 0 \; \; \textnormal{in probability as $k \rightarrow \infty$},
\]
which in turn follows from Lemma~\ref{lem:areaConv}.
\end{proof}

\subsection{Appendix: Vanishing areas}

\begin{lemma}\label{lem:areaConv}
Let $X : [0,T] \rightarrow \R^d$ be a c{\`a}dl{\`a}g semimartingale, $Y : [0,T] \to \LLL((\R^d)^{\otimes 2},\R)$ a locally bounded previsible process, and $(\D_k)_{k \geq 1}$ a sequence of deterministic partitions of $[0,T]$ such that $\lim_{k \rightarrow \infty} |\D_k| = 0$. Define $\mathbb{X}_{s,t} := \int_s^t (X_{r-}-X_s) \otimes dX_r$ as It{\^o} integrals. Then
\[
\max_{t_n \in \D_k} \Big|\sum_{j=0}^{n-1} Y_{t_j} \mathbb{X}_{t_j,t_{j+1}}\Big| \rightarrow 0 \; \; \textnormal{in probability as $k \rightarrow \infty$}.
\]
\end{lemma}

\begin{remark}
In the case $Y \equiv 1$, observe that Lemma~\ref{lem:areaConv} is an immediate consequence of the convergence in part~\ref{point:WZ1} of Proposition~\ref{prop:WongZakai} (where $\phi$ is taken as the piecewise linear interpolation).
\end{remark}

\begin{proof}
As in the proof of Lemma~\ref{lem:pvarTight}, we can decompose $X = L + K$, where $K$ is of bounded variation and $L$ is a local martingale with bounded jumps (e.g.,~\cite[Prop.~4.17, p.~42]{JacodShiryaev03}). 
Using the localizing sequence
\[
\tau_m = \inf \{ t \in [0,T] \mid |Y_t| + |L_t| + |K_t| > m \},
\]
we may further suppose that $Y$ and $L$ are uniformly bounded and that $X$ is bounded on $[0,\tau_m)$ by $m > 0$ and is constant on $[\tau_m, T]$. We now write
\[
\mathbb{X}_{s,t} = \int_{s}^t (X_{r-} - X_s)(dL_r + dK_r).
\]

For a fixed c{\`a}dl{\`a}g sample path $X$, for every $\varepsilon > 0$ we can find $r > 0$ sufficiently small and a partition $\calP = (s_0,\ldots,s_m)$ so that $|s_{i+1} - s_i| > r$ and
\[
\sup_{u,v \in [s_{i},s_{i+1})} \norms{X_{u,v}} < \varepsilon.
\]
Then whenever $|\D| < r/2$, for every $t_n \in D$, there exists at most one $s_i \in \calP$ such that $s_i \in [t_n,t_{n+1}]$.

Since $K$ is a process of finite variation, if $s_i \in [t_n,t_{n+1}]$ then
\[
\Big|\int_{t_n}^{t_{n+1}}(X_{s-} - X_{t_n})dK_s\Big| \leq \varepsilon|K|_{\onevar;[t_n,s_i]} + 2m |K|_{\onevar;[s_i,t_{n+1}]},
\]
and if no $s_i \in \calP$ is in $[t_n,t_{n+1}]$, then the upper bound is $\varepsilon|K|_{\onevar;[t_n,t_{n+1}]}$. Denoting by $[t_n,t_{n+1}]$ the interval in $\D_k$ containing $s_i \in \calP$, we then have
\[
\sum_{t_n \in \D_k} \Big|Y_{t_n}\int_{t_n}^{t_{n+1}} (X_{s-} - X_{t_n}) dK_s\Big| \leq C|Y|_\infty\Big(\varepsilon|K|_{\onevar} + 2m \sum_{s_i \in \calP} |K|_{\onevar;[s_i,t_{n+1}]}\Big),
\]
from which it follows that the LHS converges to zero a.s. as $k \rightarrow \infty$. It remains to show that
\begin{equation}\label{eq:LMarts}
\max_{t_n \in \D_k} \Big|\sum_{j=0}^{n-1} Y_{t_j} \int_{t_j}^{t_{j+1}} (X_{s-} - X_{t_n})dL_s \Big| \rightarrow 0
\end{equation}
in probability as $k \rightarrow \infty$.
By It{\^o} isometry
\[
\E\Big[\Big(Y_{t_n}\int_{t_n}^{t}(X_{s-} - X_{t_n})dL_s\Big)^2\Big] \leq C \E\Big[|Y_{t_n}|^2\int_{t_n}^{t} |X_{s-} - X_{t_n}|^2 d|[L]_s|\Big].
\]
As before, if $s_i \in \calP$ is in $[t_n,t_{n+1}]$, then
\[
\int_{t_n}^{t}|X_{s-} - X_{t_n}|^2d|[L]_s| \leq \varepsilon^2 |[L]_{t_n,s_i}| + (2m)^2|[L]_{s_i,t_{n+1}}|,
\]
and if no $s_i \in \calP$ is in $[t_n,t_{n+1}]$, then the upper bound is $\varepsilon^2[L]_{t_n,t_{n+1}}$. Hence
\[
\sum_{t_n \in \D_k} |Y_{t_n}|^2 \int_{t_n}^{t_{n+1}} |X_{s-} - X_{t_n}|^2 d|[L]_s| \leq C|Y|^2_{\infty} \Big(\varepsilon^2 |[L]_{\infty}| + (2m)^2\sum_{s_i \in \calP} |[L]_{s_i,t_{n+1}}|\Big),
\]
from which it follows that the LHS converges to zero a.s. as $k \rightarrow \infty$. As $L$ is bounded (so in particular bounded in $L^2$), we obtain by dominated convergence
\[
\E\Big[\sum_{t_n\in \D_k}\Big( Y_{t_n} \int_{t_n}^{t_{n+1}} (X_{s-}-X_{t_n})dL_s \Big)^2\Big] \leq C \E\Big[ \sum_{t_n\in \D_k} |Y_{t_n}|^2 \int_{t_n}^{t_{n+1}} |X_{s-} - X_{t_n}|^2 d|[L]_s|\Big]\rightarrow 0.
\]
Finally, applying the classical BDG inequality to the discrete-time martingale
\[
\sum_{j=0}^n Y_{t_j} \int_{t_j}^{t_{j+1}} (X_{s-}-X_{t_n})dL_s,
\]
we see that~\eqref{eq:LMarts} holds in $L^2$, and thus in probability as desired.
\end{proof}

\section{Beyond semimartingales}\label{sec:Beyond}

We have seen in the previous section that (general) multi-dimensional
semimartingales give rise to (c{\`a}dl{\`a}g) geometric $p$-rough paths. Marcus lifts of general semimartingales provide us with concrete and important examples of driving
signals for canonical RDEs, providing a decisive and long-awaited~\cite{Williams01} rough path view on classical stochastic differential equations with jumps. We now discuss several examples to which the theory of Section~\ref{sec:canonicalRDEs} can be applied which fall outside the scope of classical semimartingale theory.

\subsection{Semimartingales perturbed by paths of finite $q$-variation}

Keeping focus on $\R^{d}$-valued processes and their (canonical) lifts, we remark that any process with a decomposition $X = Y+B$, where $Y$ is a semimartingale and $B$ is a process with finite $q$-variation for some $q < 2$, admits a canonical Marcus lift given by $\X = T_B(\Y)$. Note that due to the deterministic nature of the (Young) integrals used to construct $T_B(\Y)$, we require no adaptedness assumptions on $B$.
We summarize the existence of the lift in the following proposition, which is an immediate consequence of Proposition~\ref{prop:translation}.

\begin{proposition}\label{prop:qvarPerturb}
Let $1 \leq q < 2$ and consider an $D^{\qvar}([0,T],\R^d)$-valued random variable $B$, and an $\R^d$-valued semimartingale $Y$. Write $A_{Y}$ for the area of $Y$ and
define its Marcus lift $\mathbf{Y}=\exp \left( Y+A_{Y}\right) $.
Then the process $X := Y+B$
admits a canonical lift, given by $\X=T_{B}\Y$, which is a Marcus-like geometric $p$-rough path for any $p > 2$.
\end{proposition}

We mention that the class of paths with such a decomposition contains some well-studied processes.

\begin{example}[PII]
The important class of processes with independent increments (PII) goes beyond semimartingale theory. In fact, every such process $X$ can be decomposed (non-uniquely) as $X=Y+B$,
where $Y$ is a PII\ and a semimartingale and $B$ is a deterministic c{\`a}dl{\`a}g
path. Moreover, $X$ is a semimartingale if and only if $B$ has finite
variation on compacts. Provided that the process $B$ has finite $q$-variation for some $q < 2$, we immediately see that $X$ admits a lift to a (c{\`a}dl{\`a}g) geometric $p$-rough path for any $p>2$.

We note that there is a natural interest in differential equations driven by PIIs (under the assumptions of Proposition~\ref{prop:qvarPerturb} this is meaningful!) since the resulting (pathwise) solutions $Z$ to canonical RDEs
\[
dZ = V(Z) \diamond d\X
\]
will be (time-inhomogeneous) Markov processes. A further study and characterization of such processes seems desirable. (For instance, they may not be nicely characterized by their generator. Consider the case $V\equiv 1, X = B \in C^{q-var}([0,T],\R) \setminus C^1([0,T],\R)$ with $q \in (1,2)$.)
\end{example}

\subsection{Markovian and Gaussian c{\`a}dl{\`a}g rough paths}

One can use Dirichlet forms to construct Markovian rough paths which are not lifts of semimartingales. In the continuous setting this has been developed in detail in~\cite{FrizVictoir08}. Including a non-local term in the Dirichlet form will allow to extend this construction to the jump case, but we will not investigate this here.  
We also note that Gaussian c{\`a}dl{\`a}g rough paths can be constructed; as in the continuous theory, the key condition is finite $\rho$-variation of the covariance (cf.~\cite[Sec.~10.2]{FrizHairer14}) but without assuming its continuity.

\subsection{Group-valued processes}
This point was already made in the context of L\'{e}vy rough paths~\cite{FrizShekhar17,Chevyrev15}, which substantially generalizes the notion of the Marcus lift of an $\R^d$-valued L\'{e}vy process (such processes, for example, arise naturally as limits of stochastic flows, see~\cite[Sec.~5.3.1]{Chevyrev15}). In the
same spirit, one can defined ``genuine semimartingale rough paths'' as $G^{N}(\R^d)$-valued process with local characteristics modelled after L\'{e}vy (rough
path) triplets. 
(Remark that the Lie group $G^{N}(\R^d)$ is, in particular, a differentiable manifold so that the theory of manifold-valued semimartingales applies. The issue is to identify those which constitute (geometric) rough paths, which can be done analyzing the local characteristics, as was done in the L{\'e}vy case in the afore-mentioned papers. In the same spirit, the afore-mentioned Dirichlet-form construction also extends to the group setting.)

\bibliographystyle{imsart-number}
\bibliography{AllRefsArxiv}

\end{document}